\documentclass[10pt]{amsart}
\usepackage{amsmath,amsthm,fixltx2e}
\usepackage{color}
\usepackage{amsxtra}
\usepackage{amscd}
\usepackage{amsfonts}
\usepackage{amssymb}
\usepackage{eucal}

\usepackage[all]{xy}

\newtheorem{theorem}{Theorem}[section]
\newtheorem{cor}[theorem]{Corollary}
\newtheorem{lem}[theorem]{Lemma}
\newtheorem{prop}[theorem]{Proposition}

\newtheorem{thm}[theorem]{Theorem}
\newtheorem{claim}[theorem]{Claim}
\newtheorem{rem}[theorem]{Remark}
\newtheorem{ex}[theorem]{Example}

\newcommand{\nc}{\newcommand}

\nc\ol{\overline} \nc\ul{\underline} \nc\wt{\widetilde}
\nc{\z}{\zeta}

\nc{\ZZ}{{\mathbb Z}} \nc{\NN}{{\mathbb N}} \nc{\CC}{{\mathbb C}}

\nc{\A}{{\mathbb A}}  \nc\U{U}

\nc{\F}{{\mathcal F}} \nc{\N}{{\mathcal N}} \nc{\Aa}{{\mathcal A}}

\DeclareMathOperator{\gr}{\mathrm{gr}}
\DeclareMathOperator{\Gr}{\mathrm{Gr}}

\DeclareMathOperator{\diag}{\mathrm{diag}}
\DeclareMathOperator{\tr}{\mathrm{tr}}

\DeclareMathOperator{\Sym}{\mathrm{Sym}}
\DeclareMathOperator{\Ker}{\mathrm{Ker}}
\DeclareMathOperator{\Aut}{\mathrm{Aut}}
\DeclareMathOperator{\Cas}{\mathrm{Cas}}

\DeclareMathOperator{\End}{\mathrm{End}}
\DeclareMathOperator{\cl}{\mathrm{cl}}
\DeclareMathOperator{\s}{\mathrm{S}}

\DeclareMathOperator{\HC}{\mathrm{HC}}

\newcommand\q{\mathfrak q}
\newcommand\h{\mathfrak h}
\newcommand\C{\mathfrak c}
\newcommand\m{\mathfrak m}

\newcommand{\zz}{\mathfrak{z}}

\newcommand{\gl}{\mathfrak{gl}}
\newcommand{\ssl}{\mathfrak{sl}}
\newcommand{\g}{\mathfrak{g}}
\newcommand{\spn}{\mathfrak{sp}}
\newcommand{\so}{\mathfrak{so}}

\newcommand{\ad}{\mathop{\rm ad}\nolimits}
\nc{\iso}{{\stackrel{\sim}{\longrightarrow}}}

\begin{document}

\author[Ivan Losev]{Ivan Losev}
 \address{Department of Mathematics, Northeastern University, 360 Huntington Avenue, Boston, MA 02115, USA}
 \email{i.loseu@neu.edu}

\author[Alexander Tsymbaliuk]{Alexander Tsymbaliuk}
 \address{Independent University of Moscow, 11 Bol'shoy Vlas'evskiy per., Moscow 119002, Russia}
 \curraddr{Department of Mathematics, MIT, 77 Massachusetts Avenue, Cambridge, MA  02139, USA}
 \email{sasha\_ts@mit.edu}

\title[Infinitesimal Cherednik algebras as $W$-algebras]
 {Infinitesimal Cherednik algebras as $W$-algebras}

\begin{abstract}
  In this article we establish an isomorphism between universal infinitesimal Cherednik algebras
 and $W$-algebras for Lie algebras of the same type and \textit{1-block} nilpotent elements.
 As a consequence we obtain some fundamental results about infinitesimal Cherednik algebras.
\end{abstract}

\maketitle

\begin{center}
 \emph{Dedicated to Evgeny Borisovich Dynkin, on his 90th birthday.}
\end{center}

\section*{Introduction}
  This paper is aimed at the identification of two algebras of seemingly different nature.
 The first, finite $W$-algebras, are algebras constructed from a pair $(\g, e)$, where $e$ is a nilpotent element
 of a finite dimensional simple Lie algebra $\g$. Their theory has been extensively studied during the last decade.
 For the related references see, for example, reviews~\cite{L8,W} and articles~\cite{BGK,BK1,BK2,GG,L1,L5,L6,P1,P2}.

  The second class of algebras we consider in this paper are the so called infinitesimal Cherednik algebras of type $\gl_n$ and $\spn_{2n}$,
 introduced in~\cite{EGG}. These are certain continuous analogues of the rational Cherednik algebras
 and in the case of $\gl_{n}$ are deformations of the universal enveloping algebra $\U(\ssl_{n+1})$.
 In both cases we call $n$ the \emph{rank} of an algebra. The theory of those algebras is less developed, while the main references there
 are:~\cite{EGG,T1,T2,DT}.

\medskip
 This paper is organized in the following way:

\medskip
 $\bullet$
  In Section 1, we recall the definitions of infinitesimal Cherednik algebras $H_\z(\gl_n),\ H_\z(\spn_{2n})$, and introduce their
 modified versions, called the universal length $m$ infinitesimal Cherednik algebras.
 We also recall the definitions and basic results about the finite $W$-algebras $\U(\g,e)$.

\medskip
 $\bullet$
  In Section 2, we prove our main result, establishing an abstract isomorphism of $W$-algebras $\U(\ssl_{n+m},e_m)$ (respectively $\U(\spn_{2n+2m},e_m)$)
 with the universal infinitesimal Cherednik algebras $H_m(\gl_n)$ (respectively $H_m(\spn_{2n})$).

\medskip
 $\bullet$
  In Section 3, we establish explicitly a Poisson analogue of the aforementioned isomorphism.
 As a result we deduce two claims needed to carry out the arguments of the previous section.

\medskip
 $\bullet$
  In Section 4, we derive several important consequences about algebras $H_\z(\gl_n),\ H_\z(\spn_{2n})$.
 This clarifies some lengthy computations from~\cite{T1,T2,DT} and proves new results.
 Using the results of~\cite[Section 3]{DT}, about the \emph{Casimir element} of $H_\z(\gl_n)$, we
 determine the aforementioned isomorphism $H_m(\gl_n)\iso \U(\ssl_{n+m},e_m)$ explicitly.

\medskip
 $\bullet$
  In Section 5, we recall the machinery of completions of the graded deformations of Poisson
 algebras, developed by the first author in~\cite{L1}. This provides the decomposition theorem for the completions of infinitesimal Cherednik algebras.
 This is analogous to a result by Bezrukavnikov and Etingof (\cite[Theorem 3.2]{BE}) in the theory of Rational Cherednik algebras.

\medskip
 $\bullet$
  In the Appendix, we provide some computations.

\subsection*{Acknowledgments} We are grateful to Pavel Etingof for numerous stimulating discussions. The first author was supported
by the NSF under grants DMS-0900907, DMS-1161584.


\section{Basic definitions}
\subsection{Infinitesimal Cherednik algebras of $\gl_n$}

  We recall the definition of the infinitesimal Cherednik algebras $H_\z(\gl_n)$ following~\cite{EGG}.
 Let $V_n$ and $V_n^*$ be the basic representation of $\gl_n$ and its dual. Choose a basis $\{y_i\}_{1\leq i\leq n}$ of $V_n$ and let
 $\{x_i\}_{1\leq i \leq n}$ denote the dual basis of $V_n^*$. For any $\gl_n$-invariant pairing $\z:V_n \times V_n^{*} \to \U(\gl_n)$,
 define an algebra $H_\z(\gl_n)$ as the quotient of the semi-direct product algebra $\U(\gl_n)\ltimes T(V_n \oplus V_n^*)$ by the
 relations $[y,x]=\z(y,x)$ and $[x,x']=[y,y']=0$ for all $x,x' \in V_n^{*}$ and $y,y' \in V_n$. Consider an algebra filtration on
 $H_\z(\gl_n)$ by setting $\deg(V_n)=\deg(V_n^*)=1$ and $\deg (\gl_n)=0$.

\medskip
\begin{defn}
  We say that $H_\z(\gl_n)$ \emph{satisfies the PBW property} if the natural surjective map
 $\U(\gl_n)\ltimes S(V_n\oplus V_n^*) \twoheadrightarrow \mathrm{gr} H_\z(\gl_n)$ is an isomorphism, where $S$ denotes the symmetric algebra.
 We call these $H_\z(\gl_n)$ the \emph{infinitesimal Cherednik algebras of $\gl_n$}.
\end{defn}

\medskip
  It was shown in \cite[Theorem 4.2]{EGG}, that the PBW property holds for $H_\z(\gl_n)$ if and only if $\z=\sum_{j=0}^k \z_j r_j$ for some
 nonnegative integer $k$ and $\z_j \in \mathbb{C}$, where $r_j(y,x)\in \U(\gl_n)$ is the symmetrization of $\alpha_j(y,x)\in S(\gl_n)\simeq \CC[\gl_n]$
 and $\alpha_j(y,x)$ is defined via the expansion
   $$(x,(1-\tau A)^{-1} y)\det (1-\tau A)^{-1}=\sum_{j\geq 0}{\alpha_j(y,x)(A)\tau^j},\ \ \ A\in \gl_n.$$

 Let us define the \emph{length} of such $\z$ by $l(\z):=\min\{m\in \ZZ_{\geq -1}|\ \z_{\geq m+1}=0\}$.

\begin{ex}\label{basic1} \cite[Example 4.7]{EGG}
  If $l(\z)=1$ then $H_\z(\gl_n)\cong \U(\ssl_{n+1})$. Thus, for an arbitrary $\z$, we can regard $H_\z(\gl_n)$ as a deformation of $\U(\ssl_{n+1})$.
\end{ex}

  One interesting problem is to find deformation parameters $\z$ and $\z'$ of the above form with $H_\z(\gl_n)\simeq H_{\z'}(\gl_n)$.
 Even for $n=1$ (when $H_\z(\gl_1)$ are simply the \emph{generalized Weyl algebras}),
 the answer to this question (given in~\cite{BJ}) is quite nontrivial.
  Instead, we will look only for the filtration preserving isomorphisms, where both algebras are endowed with the \emph{$N$-th standard filtration}
 $\{\F^{(N)}_{\bullet}\}$. Those are induced from the grading on $T(\gl_n\oplus V_n\oplus V_n^*)$ with $\deg(\gl_n)=2$ and
 $\deg(V_n\oplus V_n^*)=N$, where $N>l(\z)$.
  For $N\geq \max\{l(\z)+1,l(\z')+1,3\}$ we have the following result (see Appendix A for a proof):

\begin{lem}\label{less_parameters}

 (a) $N$-standardly filtered algebras $H_{\z}(\mathfrak{gl}_n)$ and $H_{\z'}(\mathfrak{gl}_n)$ are isomorphic if and only if there exist
 $\lambda\in \CC, \theta\in \CC^*, s\in \{\pm\}$ such that $\z'=\theta\varphi_\lambda(\z^{s})$, where

$\bullet$
 $\varphi_\lambda:\U(\gl_n)\overset{\sim}\longrightarrow \U(\gl_n)$ is an isomorphism defined by $\varphi_\lambda(A)=A+\lambda\cdot \tr A$ for any $A\in \gl_n$,

$\bullet$
 for $\z=\z_0r_0+\z_1r_1+\z_2r_2+\ldots$ we define $\z^{-}:=\z_0r_0-\z_1r_1+\z_2r_2-\ldots,\ \z^{+}:=\z$.

\medskip
\noindent
 (b) For any length $m$ deformation $\z$, there is a length $m$ deformation $\z'$ with $\z'_m=1,\ \z'_{m-1}=0$, such that algebras
 $H_{\z}(\mathfrak{gl}_n)$ and $H_{\z'}(\mathfrak{gl}_n)$ are isomorphic as filtered algebras.
\end{lem}

\subsection{Infinitesimal Cherednik algebras of $\spn_{2n}$}

  Let $V_{2n}$ be the standard $2n$-dimensional representation of $\spn_{2n}$ with a symplectic form $\omega$. Given any $\spn_{2n}$-invariant
 pairing $\z: V_{2n}\times V_{2n} \rightarrow \U(\spn_{2n})$ we define an algebra
   $H_\z(\spn_{2n}):=\U(\spn_{2n})\ltimes T(V_{2n})/([x,y]-\z(x,y) |x,y \in V_{2n})$.
 It has a filtration induced from the grading $\deg (\spn_{2n})=0,\ \deg(V_{2n})=1$ on $T(\spn_{2n}\oplus V_{2n})$.

\medskip
\begin{defn}
 Algebra $H_\z(\spn_{2n})$ is referred to as the \emph{infinitesimal Cherednik algebra of $\spn_{2n}$} if it satisfies the \emph{PBW property}:
 $\U(\spn_{2n})\ltimes S(V_{2n}) \overset{\sim}\longrightarrow \mathrm{gr} H_\z(\spn_{2n})$.
\end{defn}

\medskip
  It was shown in~\cite[Theorem 4.2]{EGG}, that $H_\z(\spn_{2n})$ satisfies the PBW property if and only if $\z=\sum_{j=0}^k \z_jr_{2j}$
 for some nonnegative integer $k$ and $\z_j\in \CC$, where $r_{2j}(x,y)\in \U(\spn_{2n})$ is the symmetrization of
 $\beta_{2j}(x,y)\in S(\spn_{2n})\simeq \CC[\spn_{2n}]$ and $\beta_{2j}(x,y)$ is defined via the expansion
   $$\omega(x,(1-\tau^2 A^2)^{-1} y) \det(1-\tau A)^{-1}=\sum_{j\geq 0}{\beta_{2j}(x,y)(A)\tau^{2j}},\ \ \ A\in \spn_{2n}.$$

\noindent
 Similarly to the $\gl_n$-case, we define the \emph{length} of such $\z$ by $l(\z):=\min\{m\in \ZZ_{\geq -1}|\ \z_{\geq m+1}=0\}$.

\begin{ex} \label{basic2} \cite[Example 4.11]{EGG}
  For $\z_0\ne 0$ we have $H_{\z_0r_0}(\spn_{2n})\cong \U(\spn_{2n})\ltimes W_n$, where $W_n$ is the $n$-th Weyl algebra.
 Thus, $H_\z (\spn_{2n})$ can be regarded as a deformation of $\U(\spn_{2n})\ltimes W_n$.
\end{ex}

  For any $N> 2l(\z)$, we introduce the $N$-th standard filtration $\{\F^{(N)}_{\bullet}\}$ on $H_\z (\spn_{2n})$ by
 setting $\deg(\spn_{2n})=2,\ \deg(V_{2n})=N$. The following result is analogous to Lemma~\ref{less_parameters}:

\begin{lem}\label{less_parameters2}
  For $N\geq \max\{2l(\z)+1,2l(\z')+1,3\}$, the $N$-standardly filtered algebras $H_\z(\spn_{2n})$ and $H_{\z'}(\spn_{2n})$ are isomorphic
 if and only if there exists $\theta\in \CC^*$ such that $\z'=\theta \z$.
\end{lem}

\subsection{Universal algebras $H_m(\gl_n)$ and $H_m(\spn_{2n})$}

  It is natural to consider a version of those algebras with $\z_j$ being independent central variables.
 This motivates the following notion of the universal length $m$ infinitesimal Cherednik algebras.

\medskip
\begin{defn}
 The \emph{universal length $m$ infinitesimal Cherednik algebra} $H_m(\gl_n)$ is the
 quotient of $\U(\gl_n)\ltimes T(V_n\oplus V_n^*)[\z_0,\ldots,\z_{m-2}]$ by the relations
  $$[x,x']=0,\ [y,y']=0,\ [A,x]=A(x),\ [A,y]=A(y),\ [y,x]=\sum_{j=0}^{m-2} \z_j r_j(y,x)+r_m(y,x),$$
 where $x,x'\in V_n^*,\ y,y'\in V_n,\ A\in \gl_n$ and $\{\z_j\}_{j=0}^{m-2}$ are central. The filtration is induced from the grading on
 $T(\gl_n\oplus V_n\oplus V_n^*)[\z_0,\ldots,\z_{m-2}]$ with $\deg(\gl_n)=2,\ \deg(V_n\oplus V_n^*)=m+1,\  \deg(\z_i)=2(m-i)$
 (the latter is chosen in such a way that $\deg(\z_jr_j)=2m$ for all $j$).
\end{defn}
\medskip

  Algebra $H_m(\gl_n)$ is free over $\CC[\z_0,\ldots,\z_{m-2}]$ and $H_m(\gl_n)/(\z_0-c_0,\ldots,\z_{m-2}-c_{m-2})$ is the usual
 infinitesimal Cherednik algebra $H_{\z_c}(\gl_n)$ with $\z_c=c_0r_0+\ldots+c_{m-2}r_{m-2}+r_m$. In fact, for odd $m$,
 $H_m(\gl_n)$ can be viewed as a universal family of length $m$ infinitesimal Cherednik algebras of $\gl_n$, while for even
 $m$, there is an action of $\ZZ/2\ZZ$ we should quotient by.
  \footnote{\ This follows from our proof of Lemma~\ref{less_parameters}.}

\begin{rem}\label{PBW_updated}
  One can consider all possible quotients
      $$\U(\gl_n)\ltimes T(V_n\oplus V_n^*)[\z_0,\ldots,\z_{m-2}]/([x,x'], [y,y'],[A,x]-A(x), [A,y]-A(y), [y,x]-\eta(y,x)),$$
 with a $\gl_n$-invariant pairing $\eta:V_n\times V_n^*\rightarrow \U(\gl_n)[\z_0,\ldots,\z_{m-2}]$, such that $\deg(\eta(y,x))\leq 2m$.
  Such a quotient satisfies a PBW property if and only if $\eta(y,x)=\sum_{i=0}^{m}{\eta_i(\z_0,\ldots,\z_{m-2})r_i(y,x)}$
 with $\deg(\eta_i(\z_0,\ldots,\z_{m-2}))\leq 2(m-i)$ (this is completely analogous to~\cite[Theorem 4.2]{EGG}).
\end{rem}

 We define the universal version of $H_\z(\spn_{2n})$ in a similar way:

\medskip
\begin{defn}
  The \emph{universal length $m$ infinitesimal Cherednik algebra} $H_m(\spn_{2n})$ is defined as
     $$H_m(\spn_{2n}):=\U(\spn_{2n})\ltimes T(V_{2n})[\z_0,\ldots,\z_{m-1}]/([A,x]-A(x), [x,y]-\sum_{j=0}^{m-1} \z_j r_{2j}(x,y)-r_{2m}(x,y)),$$
 where $A\in \spn_{2n},\ x,y\in V_{2n}$ and $\{\z_i\}_{i=0}^{m-1}$ are central. The filtration is induced from the grading on
 $T(\spn_{2n}\oplus V_{2n})[\z_0,\ldots,\z_{m-1}]$ with $\deg(\spn_{2n})=2,\ \deg(V_{2n})=2m+1$ and $\deg(\z_i)=4(m-i)$.
\end{defn}

\medskip

  The algebra $H_m(\spn_{2n})$ is free over $\CC[\z_0,\ldots,\z_{m-1}]$ and $H_m(\spn_{2n})\slash(\z_0-c_0,\ldots,\z_{m-1}-c_{m-1})$ is the
 usual infinitesimal Cherednik algebra $H_{\z_c}(\spn_{2n})$ for $\z_c=c_0r_0+\ldots+c_{m-1}r_{2(m-1)}+r_{2m}$.
 In fact, the algebra $H_m(\spn_{2n})$ can be viewed as a universal family of length $m$ infinitesimal Cherednik algebras of $\spn_{2n}$,
 due to Lemma~\ref{less_parameters2}.

\begin{rem}\label{PBW_updated2}
  Analogously to Remark~\ref{PBW_updated}, the result of~\cite[Theorem 4.2]{EGG}, generalizes straightforwardly to the case of
 $\spn_{2n}$-invariant pairings $\eta:V_{2n}\times V_{2n}\rightarrow \U(\spn_{2n})[\z_0,\ldots,\z_{m-1}]$.
%
\end{rem}

\subsection{Poisson counterparts of $H_\z(\g)$ and $H_m(\g)$}
\label{Poisson counterparts}
 Following~\cite{DT}, we introduce the Poisson algebras $H_m^{\cl}(\g)$ for $\g$ being $\gl_n$ or $\spn_{2n}$.

  As algebras these are $S(\gl_n\oplus V_n\oplus V_n^*)[\z_0,\ldots,\z_{m-2}]$ (respectively $S(\spn_{2n}\oplus V_{2n})[\z_0,\ldots,\z_{m-1}]$)
 with a Poisson bracket $\{\cdot,\cdot\}$ modeled after the commutator $[\cdot,\cdot]$ from the definition of $H_m(\g)$, so
 that $\{y,x\}=\alpha_m(y,x)+\sum_{j=0}^{m-2}{\z_j\alpha_j(y,x)}$ (respectively $\{x,y\}=\beta_{2m}(x,y)+\sum_{j=0}^{m-1}{\z_j\beta_{2j}(x,y)}$).
  Their quotients $H_m^{\cl}(\gl_n)/(\z_0-c_0,\ldots,\z_{m-2}-c_{m-2})$ and $H_m^{\cl}(\spn_{2n})/(\z_0-c_0,\ldots,\z_{m-1}-c_{m-1})$,
 are the Poisson infinitesimal Cherednik algebras $H_{\z_c}^{\cl}(\gl_n)$ ($\z_c=c_0\alpha_0+\ldots+c_{m-2}\alpha_{m-2}+\alpha_m$) and
 $H_{\z_c}^{\cl}(\spn_{2n})$ ($\z_c=c_0\beta_0+\ldots+c_{m-1}\beta_{2m-2}+\beta_{2m}$) from~\cite[Sections 5 and 7]{DT} respectively.

 Let us describe the Poisson centers of the algebras $H_m^{\cl}(\gl_n)$ and $H_m^{\cl}(\spn_{2n})$.

\noindent
  For $\g=\gl_n$ and $1\leq k\leq n$ we define an element $\tau_k\in H_m^{\cl}(\g)$ by $\tau_k:=\sum_{i=1}^n x_i\{\wt{Q}_k,y_i\}$,
 where $1+\sum_{j=1}^n{\wt{Q}_jz^j}=\det(1+zA)$. We set $\z(w):=\sum_{i=0}^{m-2}\z_iw^i+w^m$ and define $c_i\in S(\gl_n)$ via
 $$c(t)=1+\sum_{i=1}^n(-1)^ic_it^i:=\mathrm{Res}_{z=0}\z(z^{-1})\frac{\det(1-tA)}{\det(1-zA)}\frac{z^{-1}dz}{1-t^{-1}z}.$$

\noindent
  For $\g=\spn_{2n}$ and $1\leq k\leq n$ we define an element $\tau_k\in H_m^{\cl}(\g)$ by $\tau_k:=\sum_{i=1}^{2n}\{\wt{Q}_k,y_i\}y_i^*$,
 where $1+\sum_{j=1}^n{\wt{Q}_jz^{2j}}=\det(1+zA)$, while $\{y_i\}_{i=1}^{2n}$ and $\{y_i^*\}_{i=1}^{2n}$ are the dual bases of $V_{2n}$, that is $\omega(y_i,y_j^*)$=1.
 We set $\z(w):=\sum_{i=0}^{m-1}\z_iw^i+w^m$ and define $c_i\in S(\spn_{2n})$ via
 $$c(t)=1+\sum_{i=1}^n
 c_it^{2i}:=2\mathrm{Res}_{z=0}\z(z^{-2})\frac{\det(1-tA)}{\det(1-zA)}\frac{z^{-1}dz}{1-t^{-2}z^2}.$$
  The following result is a straightforward generalization of~\cite[Theorems 5.1 and 7.1]{DT}:

\begin{thm}\label{Poisson-center}
 Let $\zz_{\mathrm{Pois}}(A)$ denote the Poisson center of the Poisson algebra $A$. We have:

 (a) $\zz_{\mathrm{Pois}}(H_m^{\cl}(\gl_n))$ is a polynomial algebra in free generators $\z_0,\ldots,\z_{m-2},\tau_1+c_1,\ldots,\tau_n+c_n$;

 (b) $\zz_{\mathrm{Pois}}(H_m^{\cl}(\spn_{2n}))$ is a polynomial algebra in free generators $\z_0,\ldots,\z_{m-1}, \tau_1+c_1,\ldots,\tau_n+c_n$.
\end{thm}

\subsection{$W$-algebras}
 Here we recall finite $W$-algebras following~\cite{GG}.

  Let $\g$ be a finite dimensional simple Lie algebra over $\CC$ and $e\in\g$ be a nonzero nilpotent element.
 We identify $\g$ with $\g^*$ via the Killing form $(\ ,\ )$. Let $\chi$ be the element of $\g^*$ corresponding to $e$ and $\zz_{\chi}$
 be the stabilizer of $\chi$ in $\g$ (which is the same as the centralizer of $e$ in $\g$). Fix an $\mathfrak{sl}_2$-triple $(e,h,f)$ in $\g$.
 Then $\zz_\chi$ is $\ad(h)$--stable and the eigenvalues of $\ad(h)$ on $\zz_{\chi}$ are nonnegative integers.

  Consider the $\ad(h)$--weight grading on $\g=\bigoplus_{i\in\ZZ}\g(i)$, that is, $\g(i):=\{\xi\in\g| [h,\xi]=i\xi\}$.
 Equip $\g(-1)$ with the symplectic form $\omega_\chi(\xi,\eta):=\langle\chi,[\xi,\eta]\rangle$. Fix a Lagrangian subspace $l\subset \g(-1)$ and set
 $\m:=\bigoplus_{i\leq -2}\g(i)\oplus l\subset \g,\ \m':=\{\xi-\langle\chi,\xi\rangle,\xi\in\m\} \subset \U(\g)$.

\medskip
\begin{defn}\cite{P1,GG}
  By the $W$-algebra associated with $e$ (and $l$), we mean the algebra $\U(\g,e):=\left(\U(\g)/\U(\g)\m' \right)^{\ad \m}$
 with multiplication induced from $\U(\g)$.
\end{defn}
\medskip

  Let $\{F^{st}_{\bullet}\}$ denote the PBW filtration on $\U(\g)$, while $\U(\g)(i):=\{x\in \U(\g)| [h, x]=ix\}$.
 Define $F_k\U(\g)=\sum_{i+2j\leq k} (F^{st}_j \U(\g)\cap \U(\g)(i))$ and equip $\U(\g,e)$ with the induced filtration,
 denoted $\{F_{\bullet}\}$ and referred to as the {\it Kazhdan} filtration.

  One of the key results of~\cite{P1,GG} is a description of the associated graded algebra $\gr_{F_{\bullet}}\U(\g,e)$.
 Recall that the affine subspace $\s:=\chi+(\g/[\g,f])^*\subset\g^*$ is called the {\it Slodowy slice}.
 As an affine subspace of $\g$, the Slodowy slice $\s$ coincides with $e+\C$, where $\C=\Ker_{\g}\ad(f)$.
 So we can identify $\CC[\s]\cong \CC[\C]$ with the symmetric algebra $S(\zz_{\chi})$.
  According to~\cite[Section 3]{GG}, algebra $\CC[\s]$ inherits a Poisson structure from $\CC[\g^*]$ and is also graded with $\deg(\zz_{\chi}\cap \g(i))=i+2$.

\begin{thm}\label{gr}\cite[Theorem 4.1]{GG}
  The filtered algebra $\U(\g,e)$ does not depend on the choice of $l$ (up to a distinguished isomorphism) and
 $\gr_{F_{\bullet}} \U(\g,e)\cong \CC[\s]$ as graded Poisson algebras.
\end{thm}

\subsection{Additional properties of $W$-algebras}
\label{W properties}
 We want to describe some other properties of $\U(\g,e)$.

\medskip
 (a) Let $G$ be the adjoint group of $\g$. There is a natural action of the group $Q:=Z_G(e,h,f)$ on $\U(\g,e)$, due to~\cite{GG}.
 Let $\q$ stand for the Lie algebra of $Q$. In~\cite{P2} Premet constructed a Lie algebra embedding $\q\overset{\iota}\hookrightarrow \U(\g,e)$.
 The adjoint action of $\q$ on $\U(\g,e)$ coincides with the differential of the aforementioned $Q$-action.

\medskip
 (b) Restricting the natural map $\U(\g)^{\ad \m}\longrightarrow \U(\g,e)$ to $Z(\U(\g))$, we get an algebra homomorphism
 $Z(\U(\g))\overset{\rho}\longrightarrow Z(\U(\g,e))$, where $Z(A)$ stands for the center of an algebra $A$.
 According to the following theorem, $\rho$ is an isomorphism:
\begin{thm}\label{center}
 (a) \cite[Section 6.2]{P1} The homomorphism $\rho$ is injective.

 (b) \cite[footnote to Question 5.1]{P2} The homomorphism $\rho$ is surjective.
\end{thm}


\medskip
\section{Main Theorem}

  Let us consider $\g=\ssl_N$ or $\g=\spn_{2N}$, and let $e_m\in \g$ be a \emph{$1$-block} nilpotent element of Jordan type $(1,\ldots,1,m)$ or $(1,\ldots,1,2m)$,
 respectively. We make a particular choice of such $e_m$:

\medskip
\noindent
 $\bullet$ $e_m=E_{N-m+1,N-m+2}+\ldots+E_{N-1,N}$ in the case of $\ssl_N,\ 2\leq m\leq N$,

\medskip
\noindent
 $\bullet$
 $e_m=E_{N-m+1,N-m+2}+\ldots+E_{N+m-1,N+m}$ in the case of
 $\spn_{2N},\ 1\leq m\leq N$.
    \footnote {\ We view $\spn_{2N}$ as corresponding to the pair $(V_{2N},\omega_{2N})$, where $\omega_{2N}$ is represented by the skew symmetric
                 \textit{antidiagonal} matrix $J=(J_{ij}:=(-1)^j\delta_{i+j}^{2N+1})_{1\leq i,j\leq 2N}$. In this
                 presentation, $A=(a_{ij})\in\spn_{2N}$ if and only if $a_{2N+1-j,2N+1-i}=(-1)^{i+j+1}a_{ij}$ for any $1\leq i,j\leq 2N$.}

\medskip
  Recall the Lie algebra inclusion $\iota: \q\hookrightarrow U(\g,e)$ from Section~\ref{W properties}. In our cases:

\medskip\noindent
$\bullet$
  For $(\g,e)=(\ssl_{n+m},e_m)$, we have $\q\simeq \gl_n$. Define $\bar{T}\in U(\ssl_{n+m},e_m)$ to be the $\iota$-image of
 the identity matrix $I_n\in \gl_n$, the latter being identified with $T_{n,m}=\diag(\frac{m}{n+m},\cdots,\frac{m}{n+m},\frac{-n}{n+m},\cdots,\frac{-n}{n+m})$
 under the inclusion $\q\hookrightarrow \ssl_{n+m}$. Let $\Gr$ be the induced $\ad(\bar{T})$-weight grading on $U(\ssl_{n+m},e_m)$, with the $j$-th grading component denoted by
 $U(\ssl_{n+m},e_m)_j$.

\medskip
\noindent
 $\bullet$
  For $(\g,e)=(\spn_{2n+2m},e_m)$, we have $\q\simeq \spn_{2n}$. Define $\bar{T}':=\iota(I'_n)\in U(\spn_{2n+2m},e_m)$, where
 $I'_n=\diag(1,\ldots,1,-1,\ldots,-1)\in \spn_{2n}\simeq \q$. Let $\Gr$ be the induced $\ad(\bar{T}')$-weight grading on
 $U(\spn_{2n+2m},e_m)=\bigoplus_j U(\spn_{2n+2m},e_m)_j$.

\begin{lem}\label{fck}
  There is a natural Lie algebra inclusion $\Theta:\gl_n\ltimes V_n\hookrightarrow U(\ssl_{n+m},e_m)$ such that $\Theta\mid_{\gl_n}=\iota\mid_{\gl_n}$ and
 $\Theta(V_n)=F_{m+1}U(\ssl_{n+m},e_m)_1$.
\end{lem}

\begin{proof}
  First, choose a Jacobson-Morozov $\ssl_2$-triple $(e_m,h_m,f_m)\subset \ssl_{n+m}$ in a standard
  way. \footnote{\ That is we set $h_m:=\sum_{j=1}^m {(m+1-2j)E_{n+j,n+j}}$ and $f_m:=\sum_{j=1}^{m-1} j(m-j)E_{n+j+1,n+j}$.}
  As a vector space, $\zz_\chi\cong \gl_n\oplus V_n\oplus V_n^*\oplus \CC^{m-1}$ with
 $\gl_n=\zz_\chi(0)=\q,\ V_n\oplus V_n^*\subset \zz_\chi(m-1)$ and $\xi_j\in \zz_\chi(2m-2j-2)$.
  Here $\CC^{m-1}$ has a basis $\{\xi_{m-2-j}=E_{n+1,n+j+2}+\ldots+E_{n+m-j-1,n+m}\}_{j=0}^{m-2}$,
 $V_n\oplus V_n^*$ is embedded via $y_i \mapsto E_{i,n+m},\ x_i\mapsto E_{n+1,i}$, while
 $\gl_n\cong\ssl_n\oplus \CC\cdot I_n$ is embedded in the following way: $\ssl_n\hookrightarrow \ssl_{n+m}$ as a \emph{left-up block},
 while $I_n\mapsto T_{n,m}$.

  Under the identification $\gr_{F_{\bullet}}U(\ssl_{n+m},e_m)\simeq \CC[\s]\simeq S(\zz_\chi)$, the induced grading $\Gr '$ on $S(\zz_\chi)$ is the
 $\ad(T_{n,m})$-weight grading. Together with the above description of $\ad(h_m)$-grading on $\zz_\chi$, this implies that
 $F_mU(\ssl_{n+m},e_m)_1=0$ and that $F_{m+1}U(\ssl_{n+m},e_m)_1$ coincides with the image of the composition
 $V_n\hookrightarrow \zz_\chi\hookrightarrow S(\zz_\chi)$.
  Let $\Theta(y)\in F_{m+1}U(\ssl_{n+m},e_m)_1$ be the element whose image is identified with $y$.
 We also set $\Theta(A):=\iota(A)$ for $A\in \gl_n$. Finally, we define $\Theta:\gl_n\oplus V_n\hookrightarrow  U(\ssl_{n+m},e_m)$ by linearity.

  We claim that $\Theta$ is a Lie algebra inclusion, that is
    $$[\Theta(A),\Theta(B)]=\Theta([A,B]),\ [\Theta(y),\Theta(y')]=0,\ [\Theta(A),\Theta(y)]=\Theta(A(y)),\ \forall\ A,B\in \gl_n, y,y'\in  V_n.$$
  The first equality follows from $[\Theta(A),\Theta(B)]=[\iota(A),\iota(B)]=\iota([A,B])=\Theta([A,B])$.
 The second one follows from the observation that $[\Theta(y),\Theta(y')]\in F_{2m}\U(\g,e_m)_2$ and the only such element is $0$.
 Similarly, $[\Theta(A),\Theta(y)]\in F_{m+1}\U(\g,e_m)_1$, so that $[\Theta(A),\Theta(y)]=\Theta(y')$ for some $y'\in V_n$.
 Since $y'=\gr(\Theta(y'))=\gr([\Theta(A),\Theta(y)])=[A,y]=A(y)$, we get $[\Theta(A),\Theta(y)]=\Theta(A(y))$.
\end{proof}

 Our main result is:
\begin{thm}\label{main 1}

 (a) For $m\geq 2$, there is a unique isomorphism $\bar{\Theta}:H_m(\gl_n)\iso \U(\ssl_{n+m},e_m)$  of filtered algebras, whose restriction to
 $\ssl_n\ltimes V_n\hookrightarrow H_m(\gl_n)$ is equal to $\Theta$.

\noindent
 (b) For $m\geq 1$, there are exactly two isomorphisms $\bar{\Theta}_{(1)}, \bar{\Theta}_{(2)}:H_m(\spn_{2n})\iso \U(\spn_{2n+2m},e_m)$
     of filtered algebras such that $\bar{\Theta}_{(i)}\mid_{\spn_{2n}}=\iota\mid_{\spn_{2n}}$;
     moreover, $\bar{\Theta}_{(2)}\circ \bar{\Theta}_{(1)}^{-1}:y\mapsto -y, A\mapsto A, \z_k\mapsto \z_k$.
\end{thm}

  Let us point out that there is no explicit presentation of $W$-algebras in terms of generators and relations in general.
 Among few known cases are: (a) $\g=\gl_n$, due to~\cite{BK1}, (b) $\g\ni e$--the minimal nilpotent, due to~\cite[Section 6]{P2}.
 The latter corresponds to $(e_2,\ssl_N)$ and $(e_1,\spn_{2N})$ in our notation. We establish the corresponding isomorphisms explicitly in
 Appendix B.

\begin{proof}[Proof of Theorem~\ref{main 1}]
$\ $

\noindent
 (a) Analogously to Lemma~\ref{fck}, we have an identification $F_{m+1}U(\ssl_{n+m},e_m)_{-1}\simeq V_n^*$.
 For any $x\in V_n^*$, let $\Theta(x)\in F_{m+1}U(\ssl_{n+m},e_m)_{-1}$ be the element identified with $x\in V_n^*$.
 The same argument as in the proof of Lemma~\ref{fck} implies $[\Theta(A),\Theta(x)]=\Theta(A(x))$.

  Let $\{\wt{F}_j\}_{j=2}^{n+m}$ be the standard degree $j$ generators of
 $\CC[\ssl_{n+m}]^{\mathrm{SL}_{n+m}}\simeq S(\ssl_{n+m})^{\mathrm{SL}_{n+m}}$ (that is $1+\sum_{j=2}^{n+m}{\wt{F}_j(A)z^j}=\det(1+zA)$
 for $A\in \ssl_{n+m}$) and $F_j:=\Sym(\wt{F}_j)\in U(\ssl_{n+m})$ be the free generators of $Z(\U(\ssl_{n+m}))$.
  For all $0\leq i\leq m-2$ we set $\Theta_i:=\rho (F_{m-i})\in Z(\U(\ssl_{n+m},e_m))$.
 Then $\gr(\Theta_k)=\wt{F}_{m-k}{_{\mid_{\s}}}\equiv \xi_k \mod S(\gl_n\oplus \bigoplus_{l=k+1}^{m-2}\CC\xi_l)$,
 where $\xi_k$ was defined in the proof of Lemma~\ref{fck}.

 Let $U'$ be a subalgebra of $U(\ssl_{n+m},e_m)$, generated by $\Theta(\gl_n)$ and $\{\Theta_k\}_{k=0}^{m-2}$.
 For all $y\in V_n,\ x\in V_n^*$ we define $W(y,x):=[\Theta(y),\Theta(x)]\in F_{2m}\U(\ssl_{n+m},e_m)_0\subset U'$.
 Let us point out that equalities $[\Theta(A),\Theta(x)]=\Theta([A,x]),\
 [\Theta(A),\Theta(y)]=\Theta([A,y])$ (for all $A\in \gl_n, y\in V_n, x\in V_n^*$) imply the $\gl_n$-invariance of $W:V_n\times V_n^*\to U'\simeq U(\gl_n)[\Theta_0,\ldots.\Theta_{m-2}]$.

 By Theorem~\ref{gr}, $U(\ssl_{n+m},e_m)$ has a basis formed by the ordered monomials in
 $$\{\Theta(E_{ij}),\ \Theta(y_k),\ \Theta(x_l),\ \Theta_0, \ldots,\Theta_{m-2}\}.$$

 In particular, $U(\ssl_{n+m},e_m)\simeq U(\gl_n)\ltimes T(V_n\oplus V_n^*)[\Theta_0,\ldots,\Theta_{m-2}]/(y\otimes x-x\otimes y-W(y,x))$
 satisfies the PBW property. According to Remark~\ref{PBW_updated},
 there exist polynomials $\eta_i\in \CC[\Theta_0,\ldots,\Theta_{m-2}]$, for $0\leq i\leq m-2$, such that
 $W(y,x)=\sum{\eta_jr_j(y,x)}$ and $\deg(\eta_i(\Theta_0,\ldots,\Theta_{m-2}))\leq 2(m-i)$.
  As a consequence of the latter condition: $\eta_m, \eta_{m-1}\in \CC$.

 The following claim follows from the main result of the next section (Theorem~\ref{main 2}):

\begin{claim}\label{regular1}
 (i) The constant $\eta_m$ is nonzero,

 (ii) The polynomial $\eta_i(\Theta_0,\ldots,\Theta_{m-2})$ contains a nonzero multiple of $\Theta_i$ for any $i\leq m-2$.
\end{claim}

\noindent
 This claim implies the existence and uniqueness of the isomorphism $\bar{\Theta}:H_m(\gl_n)\iso \U(\ssl_{n+m},e_m)$ with
   $\bar{\Theta}(y_k)=\Theta(y_k)$ and $\bar{\Theta}(A)=\Theta(A)$ for $A\in \ssl_n$.

  Moreover, $\bar{\Theta}(x_k)=\eta_m^{-1}\Theta(x_k)$ and $\bar{\Theta}(I_n)=\Theta(I_n)-\frac{n\eta_{m-1}}{(n+m)\eta_m}$
 \footnote{\ The appearance of the constant $\frac{n\eta_{m-1}}{(n+m)\eta_m}$ is explained by the proof of
 Lemma~\ref{less_parameters}(b).},
 while $\bar{\Theta}(\z_k)\in \CC[\Theta_k,\ldots,\Theta_{m-2}]$.

\medskip
\noindent(b)
 Choose a Jacobson-Morozov $\ssl_2$-triple $(e_m,h_m,f_m)\subset \spn_{2n+2m}$ in a standard way.
 \footnote{\ That is $h_m:=\sum_{j=1}^{2m} (2m+1-2j) E_{n+j,n+j}$ and $f_m:=\sum_{j=1}^{2m-1} j(2m-j)E_{n+j+1,n+j}$.}
  As a vector space, $\zz_\chi\cong \spn_{2n}\oplus V_{2n}\oplus \CC^m$ with
 $\spn_{2n}=\zz_\chi(0),\ V_{2n}=\zz_\chi(2m-1)$ and $\xi_j\in \zz_\chi(4m-4j-2)$.
  Here $\CC^m$ has a basis $\{\xi_{m-k}=E_{n+1,n+2k}+\ldots+E_{n+2m-2k+1, n+2m}\}_{k=1}^m$, $V_{2n}$ is embedded via
  $$y_i\mapsto E_{i,n+2m}+(-1)^{n+i+1}E_{n+1,2n+2m+1-i},\ y_{n+i}\mapsto E_{n+2m+i,n+2m}+(-1)^{i+1}E_{n+1,n+1-i},\ i\leq n,$$
 while
 $\q=\zz_\chi(0)\simeq \spn_{2n}$ is embedded in a natural way (via four \emph{$n\times n$ corner blocks} of $\spn_{2n+2m}$).

  Recall the grading $\Gr$ on $U(\spn_{2n+2m},e_m)$. The induced grading $\Gr'$ on $\gr U(\spn_{2n+2m},e_m)$,
 is the $\ad(I'_n)$-weight grading on $S(\zz_\chi)$. The operator $\ad(I'_n)$ acts trivially on $\CC^m$,
 with even eigenvalues on $\spn_{2n}$ and with eigenvalues
 $\pm 1$ on $V^\pm_{2n}$, where $V^+_{2n}$ is spanned by $\{y_i\}_{i\leq n}$, while $V^-_{2n}$ is spanned by $\{y_{n+i}\}_{i\leq n}$.

  Analogously to Lemma~\ref{fck}, we get identifications of
 $F_{2m+1}U(\spn_{2n+2m},e_m)_{\pm 1}$ and $V^{\pm}_{2n}$. For $y\in V^\pm_{2n}$,
 let $\Theta(y)$ be the corresponding element of $F_{2m+1}U(\spn_{2n+2m},e_m)_{\pm 1}$, while for $A\in \spn_{2n}$ we set $\Theta(A):=\iota(A)$.
 We define $\Theta:\spn_{2n}\oplus V_{2n}\hookrightarrow U(\spn_{2n+2m},e_m)$ by linearity. The same reasoning as in the
 $\gl_n$-case proves that $[\Theta(A),\Theta(y)]=\Theta(A(y))$ for any $A\in \spn_{2n},y\in V_{2n}$.

  Finally, the argument involving the center goes along the same lines, so we can pick central generators $\{\Theta_k\}_{0\leq k\leq m-1}$
 such that $\gr(\Theta_k)\equiv \xi_k \mod S(\spn_{2n}\oplus \CC\xi_{k+1}\oplus\ldots\oplus\CC\xi_{m-1})$.

  Let $U'$ be the subalgebra of $U(\spn_{2n+2m},e_m)$, generated by $\Theta(\spn_{2n})$ and $\{\Theta_k\}_{k=0}^{m-1}$.
  For $x,y\in V_{2n}$, we set $W(x,y):=[\Theta(x),\Theta(y)]\in F_{4m}U(\spn_{2n+2m},e_m)_{\mathrm{even}}\subset U'$.
  The map
   $$W:V_{2n}\times V_{2n}\to U'\simeq  U(\spn_{2n})[\Theta_0,\ldots,\Theta_{m-1}]$$
  is $\spn_{2n}$-invariant.

  Since $U(\spn_{2n+2m},e_m)\simeq U(\spn_{2n})\ltimes T(V_{2n})[\Theta_0,\ldots,\Theta_{m-1}]/(x\otimes y-y\otimes x-W(x,y))$ satisfies the PBW
  property, there exist polynomials $\eta_i\in \CC[\Theta_0,\ldots,\Theta_{m-1}]$, for $0\leq i\leq m-1$, such that
 $W(x,y)=\sum{\eta_jr_{2j}(x,y)}$ and $\deg(\eta_i(\Theta_0,\ldots,\Theta_{m-1}))\leq 4(m-i)$ (Remark~\ref{PBW_updated2}).

 The following result is analogous to Claim~\ref{regular1} and will follow from Theorem~\ref{main 2} as well:
\begin{claim}\label{regular2}
 (i) The constant $\eta_m$ is nonzero,

 (ii) The polynomial $\eta_i(\Theta_0,\ldots,\Theta_{m-1})$ contains a nonzero multiple of $\Theta_i$ for any $i\leq m-1$.
\end{claim}

  This claim implies Theorem~\ref{main 1}(b), where $\bar{\Theta}_{(i)}(y)=\lambda_i\cdot\Theta(y)$ for all $y\in V_{2n}$ and $\lambda_i^2=\eta_m^{-1}$.
\end{proof}


\medskip
\section{Poisson analogue of Theorem~\ref{main 1}}

 To state the main result of this section, let us introduce more notation:

\medskip \noindent
 $\bullet$ In the contexts of $(\ssl_{n+m},e_m)$ and $(\spn_{2n+2m},e_m)$, we use $\s_{n,m}$ and $\zz_{n,m}$ instead of $\s$ and
 $\zz_\chi$.

\medskip \noindent
 $\bullet$ Let $\bar{\iota}:\gl_n\oplus V_n\oplus V_n^*\oplus \CC^{m-1}\iso \zz_{n,m}$ be the identification
           from the proof of Lemma~\ref{fck}.

\medskip \noindent
 $\bullet$ Let $\bar{\iota}:\spn_{2n}\oplus V_{2n}\oplus \CC^m\iso \zz_{n,m}$ be the identification from the proof of Theorem~\ref{main
 1}(b).

\medskip \noindent
 $\bullet$ Define $\bar{\Theta}_k=\gr(\Theta_k)\in S(\zz_{n,m})\ 0\leq k\leq m-s$,
           where $s=1$ for $\spn_{2N}$ and $s=2$ for $\ssl_N$.

\medskip \noindent
 $\bullet$ We consider the Poisson structure on $S(\zz_{n,m})$ arising from the identification
             $$S(\zz_{n,m})\cong \CC[\s_{n,m}].$$

 The following theorem can be viewed as a Poisson analogue of Theorem~\ref{main 1}:
\begin{thm}\label{main 2}
(a) The formulas
    $$\bar{\Theta}^{\cl}(A)=\bar{\iota}(A),\ \bar{\Theta}^{\cl}(y)=\bar{\iota}(y),\
      \bar{\Theta}^{\cl}(x)=\bar{\iota}(x),\ \bar{\Theta}^{\cl}(\z_k)=(-1)^{m-k}\bar{\Theta}_k$$
    define an isomorphism $\bar{\Theta}^{\cl}:H_m^{\cl}(\gl_n)\iso S(\zz_{n,m})\simeq \CC[\s_{n,m}]$ of Poisson algebras.

\medskip
(b) The formulas
    $$\bar{\Theta}^{\cl}(A)=\bar{\iota}(A),\ \bar{\Theta}^{\cl}(y)=\bar{\iota}(y)/\sqrt{2},\ \bar{\Theta}^{\cl}(\z_k)=\bar{\Theta}_k$$
    define an isomorphism $\bar{\Theta}^{\cl}:H_m^{\cl}(\spn_{2n})\iso S(\zz_{n,m})\simeq \CC[\s_{n,m}]$ of Poisson algebras.

\end{thm}

 Claims~\ref{regular1} and~\ref{regular2} follow from this theorem.

\begin{rem}
 An alternative proof of Claims~\ref{regular1} and~\ref{regular2} is
 based on the recent result of~\cite{LNS} about the universal
 Poisson deformation of $\s\cap \N$ (here $\N$ denotes the nilpotent cone of the Lie algebra $\g$). We find this argument a
 bit overkilling (besides, it does not provide precise formulas in the Poisson case).
\end{rem}

\begin{proof}[Proof of Theorem~\ref{main 2}]
$\ $

\noindent
 (a)  The Poisson algebra $S(\zz_{n,m})$ is equipped both with the Kazhdan grading and
 the internal grading $\Gr'$. In particular, the same reasoning as in the proof of Theorem~\ref{main 1}(a) implies:
 $$\{\bar{\iota}(A), \bar{\iota}(B)\}=\bar{\iota}([A,B]),\ \{\bar{\iota}(A),\bar{\iota}(y)\}=\bar{\iota}(A(y)),\
 \{\bar{\iota}(A),\bar{\iota}(x)\}=\bar{\iota}(A(x)).$$

 We set $\bar{W}(y,x):=\{\bar{\iota}(y),\bar{\iota}(x)\}$ for all $y\in V_n, x\in
 V_n^*$. Arguments analogous to those used in the proof of Theorem~\ref{main 1}(a)
 imply an existence of polynomials $\bar{\eta}_i\in \CC[\bar{\Theta}_0,\ldots,\bar{\Theta}_{m-2}]$ such that
 $\bar{W}(y,x)=\sum_j\bar{\eta}_j\alpha_j(y,x)$ and $\deg(\bar{\eta}_j(\bar{\Theta}_0,\ldots,\bar{\Theta}_{m-2}))=2(m-j)$.

  Combining this with Theorem~\ref{Poisson-center}(a) one gets that
 $$\tau'_1=\sum_i x_iy_i+\sum_{j}\bar{\eta}_j\tr S^{j+1}A$$
 is a Poisson-central element of $S(\zz_{n,m})\cong \CC[\s_{n,m}]$.

  Let $\bar{\rho}:\zz_{\mathrm{Pois}}(\CC[\ssl_{n+m}])\to \zz_{\mathrm{Pois}}(\CC[\s_{n,m}])$ be the restriction homomorphism.
 The Poisson analogue of Theorem~\ref{center} (which is, actually, much simpler) states that $\bar{\rho}$ is an isomorphism.
 In particular, $\tau'_1=c \bar{\rho}(\wt{F}_{m+1})+p(\bar{\rho}(\wt{F}_2),\ldots,\bar{\rho}(\wt{F}_m))$ for some $c\in \CC$ and a polynomial~$p$.

 Note that $\bar{\rho}(\wt{F}_i)=\bar{\Theta}_{m-i}$ for all $2\leq i\leq m$. Let us now express $\bar{\rho}(\wt{F}_{m+1})$ via the generators of
 $S(\zz_{n,m})$. First, we describe explicitly the slice $\s_{n,m}$. It
 consists of the following elements:
 $$\left\{e_m+\sum_{i,j\leq n}x_{i,j}E_{i,j}+\sum_{i\leq n}u_iE_{i,n+1}+\sum_{i\leq n}v_iE_{n+m,i}+
                   \sum_{k\leq m-1}w_kf_m^k-\frac{\sum_{i\leq n}x_{ii}}{m}\sum_{n<j\leq n+m}E_{jj}\right\},$$
 which can be explicitly depicted as follows:
 $$\s_{n,m}=\left\{X=\left(\begin{array}{ccccccccc}
             x_{1,1} & x_{1,2} & \cdots & x_{1,n} & u_1 & 0 & 0 & \cdots &0\\
             x_{2,1} & x_{2,2} & \cdots & x_{2,n} & u_2 & 0 & 0 & \cdots &0\\
             \vdots & \vdots & \ddots & \vdots & \vdots & \vdots & \vdots & \ddots & \vdots \\
             x_{n,1} & x_{n,2} & \cdots & x_{n,n} & u_n & 0 & 0 & \cdots &0\\
             0      &  0     & \cdots &    0   & \lambda & 1 & 0 & \cdots &0\\
             0      &  0     & \cdots &    0   &   \star  & \lambda & 1 & \cdots &0\\
             \vdots & \vdots & \ddots & \vdots & \vdots & \vdots & \vdots & \ddots & \vdots \\
             0      &   0    & \cdots &   0    & \star   & \star & \star  & \cdots & 1\\
             v_1    &   v_2  & \cdots & v_n    & \star   & \star & \star  & \cdots  & \lambda \\
           \end{array}\right)\right\}
$$
 For $X\in \ssl_{n+m}$ of the above form let us define $X_1\in \gl_n,\ X_2\in \gl_m$ by
 $$X_1:=\sum_{i,j\leq n}x_{i,j}E_{i,j},\ X_2:=e_m+\sum_{k\leq m-1}w_kf_m^k-\frac{x_{11}+\cdots+x_{nn}}{m}\sum_{n<j\leq n+m}E_{jj},$$
 that is, $X_1$ and $X_2$ are the left-up $n\times n$ and right-down $m\times m$ blocks of $X$, respectively.

  The following result is straightforward:
\begin{lem}\label{block 1}
Let $X, X_1, X_2$ be as above. Then:

\noindent
 (i) For $2\leq k\leq m:$ $\wt{F}_k(X)=\tr \Lambda^k (X_1)+\tr \Lambda^{k-1} (X_1) \tr \Lambda^1 (X_2)+\ldots+\tr \Lambda^k (X_2)$.

\medskip
\noindent
 (ii) We have $\wt{F}_{m+1}(X)=(-1)^m\sum u_iv_i+\tr \Lambda^{m+1} (X_1)+\tr \Lambda^{m}(X_1) \tr \Lambda^1 (X_2)+\ldots+\tr \Lambda^{m+1} (X_2)$.
\end{lem}

 Combining both statements of this lemma with the standard equality
\begin{equation}\label{1}
  \sum_{0\leq j\leq l}(-1)^j\tr S^{l-j}(X_1)\tr \Lambda^j (X_1)=0,\ \ \ \ \ \forall l\geq 1,
\end{equation}
 we obtain the following result:

\begin{lem}\label{block 2}
  For any $X\in \s_{n,m}$ we have:
 \begin{equation}\label{2}
  \wt{F}_{m+1}(X)=(-1)^m\sum u_iv_i+\sum_{2\leq j\leq m}(-1)^{m-j}\wt{F}_j(X)\tr S^{m+1-j}(X_1)+(-1)^m\tr  S^{m+1}(X_1).
\end{equation}
\end{lem}

\begin{proof}[Proof of Lemma~\ref{block 2}]
  Lemma~\ref{block 1}(i) and equality (1) imply by induction on $k$:
 $$\tr \Lambda^k (X_2)=\wt{F}_k(X)-\tr S^1 (X_1) \wt{F}_{k-1}(X)+\tr S^2 (X_1) \wt{F}_{k-2}(X)-\ldots+(-1)^k\tr S^k(X_1)\wt{F}_0(X),$$
 for all $k\leq m$, where $\wt{F}_1(X):=0,\ \wt{F}_0(X):=1$.

  Those equalities together with Lemma~\ref{block 1}(ii) imply:
 $$\wt{F}_{m+1}(X)=(-1)^m\sum{u_iv_i}+\sum_{0\leq j\leq m}{\sum_{0\leq k< m+1-j}(-1)^k\tr \Lambda^{m+1-j-k}(X_1)\tr S^k(X_1)\wt{F}_j(X)}.$$
 According to (1), we have $\sum_{k=0}^{m-j}(-1)^k\tr \Lambda^{m+1-j-k}(X_1)\tr S^k(X_1)=(-1)^{m-j}\tr S^{m+1-j}(X_1)$.
 Recalling our convention $\wt{F}_1(X)=0,\ \wt{F}_0(X)=1$, we get (2).
\end{proof}

 Identifying $\CC[\s_{n,m}]$ with $S(\zz_{n,m})$ we get
\begin{equation}\label{3}
 \bar{\rho}(\wt{F}_{m+1})=(-1)^m\left(\sum x_iy_i+\tr S^{m+1}A+\sum_{2\leq j\leq m}(-1)^j\bar{\Theta}_{m-j}\tr S^{m+1-j}A\right).
\end{equation}

 Substituting this into $\tau'_1=c \bar{\rho}(\wt{F}_{m+1})+p(\bar{\Theta}_0,\ldots,\bar{\Theta}_{m-2})$ with $\bar{\Theta}_{m-1}:=0,\ \bar{\Theta}_m:=1$, we get
 $$p(\bar{\Theta}_0,\ldots,\bar{\Theta}_{m-2})=(1-(-1)^mc)\sum_i {x_iy_i}+\sum_{0\leq j\leq
 m} {(\bar{\eta}_j(\bar{\Theta}_0,\ldots,\bar{\Theta}_{m-2})-(-1)^jc\bar{\Theta}_j)\tr S^{j+1}A}.$$

 Hence $c=(-1)^m$ and $p(\bar{\Theta}_0,\ldots,\bar{\Theta}_{m-2})=\sum_{0\leq j\leq
 m} {(\bar{\eta}_j(\bar{\Theta}_0,\ldots,\bar{\Theta}_{m-2})-(-1)^{m-j}\bar{\Theta}_j)\tr
 S^{j+1}A}$.
  According to Remark~\ref{PBW_updated}, the last equality is equivalent to
 $$\bar{\eta}_m=1,\ \bar{\eta}_{m-1}=0,\ \bar{\eta}_j(\bar{\Theta}_0,\ldots,\bar{\Theta}_{m-2})=(-1)^{m-j}\bar{\Theta}_j,\ \ \ \forall\ 0\leq j\leq m-2,\ \ p=0.$$

 This implies the statement.

\medskip
\noindent
 (b) Analogously to the previous case and the proof of Theorem~\ref{main 1}(b) we have:
 $$\{\bar{\iota}(A),\bar{\iota}(B)\}=\bar{\iota}([A,B]),\ \{\bar{\iota}(A),\bar{\iota}(y)\}=\bar{\iota}(A(y)),\
   \{\bar{\iota}(x),\bar{\iota}(y)\}=\sum \bar{\eta}_j\beta_{2j}(x,y),$$
 for some polynomials $\bar{\eta}_j\in \CC[\bar{\Theta}_0,\ldots,\bar{\Theta}_{m-1}]$ such that
 $\deg(\bar{\eta}_j(\bar{\Theta}_0,\ldots,\bar{\Theta}_{m-1}))=4(m-j)$.

  Due to Theorem~\ref{Poisson-center}(b), we get
 $\tau'_1:=\sum_{i=1}^{2n}\{\wt{Q}_1,y_i\}y_i^*-2\sum_j \bar{\eta}_j\tr S^{2j+2}A\in \zz_{\mathrm{Pois}}(S(\zz_{n,m}))$.
 In particular, $\tau'_1=c \bar{\rho}(\wt{F}_{m+1})+p(\bar{\rho}(\wt{F}_1),\ldots,\bar{\rho}(\wt{F}_m))$ for some $c\in \CC$ and a polynomial~$p$.

  Note that $\bar{\rho}(\wt{F}_k)=\bar{\Theta}_{m-k}$ for $1\leq k\leq m$. Let us now express $\bar{\rho}(\wt{F}_{m+1})$ via the generators of
 $S(\zz_{n,m})$. First, we describe explicitly the slice $\s_{n,m}$. It
 consists of the following elements:
 $$\{e_m+\bar{\iota}(X_1)+\sum_{i\leq n}v_iU_{i,n+1}+\sum_{i\leq n}v_{n+i}U_{n+2m+i,n+1}+\sum_{k\leq m}w_kf_m^{2k-1}|X_1\in \spn_{2n},\ v_i, v_{n+i}, w_k\in \CC\},$$
 where $U_{i,j}:=E_{i,j}+(-1)^{i+j+1}E_{2n+2m+1-j,2n+2m+1-i}\in \spn_{2n+2m}$.
 For $X\in \spn_{2n+2m}$ as above, we define
 $X_2:=e_m+\sum_{k\leq m}w_kf_m^{2k-1}\in \spn_{2m}$, viewed as the \emph{centered} $2m\times 2m$ block of $X$.

  Analogously to~(\ref{3}), we get
\begin{equation}\label{4}
 \bar{\rho}(\wt{F}_{m+1})=\frac{1}{4}\sum_{i=1}^{2n}\{\wt{Q}_1,y_i\}y_i^*-\tr S^{2m+2}A-\sum_{0\leq j\leq m-1}\bar{\Theta}_{j}\tr S^{2j+2}A.
\end{equation}

 Comparing the above two formulas for $\tau'_1$, we get the following equality:
$$\sum_{i=1}^{2n}\{\wt{Q}_1,y_i\}y_i^*-2\sum_j \bar{\eta}_j\tr S^{2j+2}A=c\cdot \bar{\rho}(\wt{F}_{m+1})+p(\bar{\Theta}_0,\ldots,\bar{\Theta}_{m-1}).$$
 Arguments analogous to those used in part (a) establish
 $$c=4,\ p=0,\ \bar{\eta}_m=2,\ \bar{\eta}_{j}=2\bar{\Theta}_j,\ \ \forall\ j<m.$$

 Part (b) follows.
\end{proof}

\begin{rem}
 Recalling the standard convention $\U(\g,0)=\U(\g)$ and Example~\ref{basic1}, we see that Theorem~\ref{main 1}(a) (as well as Theorem~\ref{main 2}(a))
 obviously holds for $m=1$ with $e_1:=0\in \ssl_{n+1}$.
\end{rem}

  The results of Theorems~\ref{main 1} and~\ref{main 2} can be naturally generalized to
 the case of the universal infinitesimal Hecke algebras of $\so_n$.
 However, this requires reproving some basic results about the latter algebras, similar to those of~\cite{EGG,DT}, and is discussed separately in~\cite{T}.


$\ $

\section{Consequences}

  In this section we use Theorem~\ref{main 1} to get some new (and recover some old) results about the algebras of interest.
 On the $W$-algebra side, we get presentations of $U(\ssl_n,e_m)$ and $U(\spn_{2n},e_m)$ via generators and relations
 (in the latter case there was no presentation known for $m>1$).
 We get much more results about the structure and the representation theory of infinitesimal Cherednik algebras using the corresponding results on $W$-algebras.

 Also we determine the isomorphism from Theorem~\ref{main 1}(a) basically explicitly.

\subsection{Centers of $H_m(\gl_n)$ and $H_m(\spn_{2n})$}\label{centers}

   We set $s=2$ for $\g=\ssl_N$ and $s=1$ for $\g=\spn_{2N}$.
 Recall the elements $\{\wt{F}_i\}_{i=s}^{N}$, where $\deg(\wt{F}_i)=(3-s)i$.
 These are the free generators of the Poisson center $\zz_{\mathrm{Pois}}(S(\g))$.
 The Lie algebra $\q=\zz_\g(e,h,f)$ from Section~\ref{W properties} equals $\gl_n$ for $(\g,e)=(\ssl_{n+m},e_m)$ and
 $\spn_{2n}$ for $(\g,e)=(\spn_{2n+2m},e_m)$. Thus $\{\wt{Q}_j\}$ from Section~\ref{Poisson counterparts} are the free generators of $\zz_{\mathrm{Pois}}(S(\q))$
 and $Q_j:=\Sym(\wt{Q}_j)$ are the free generators of $Z(\U(\q))$.

 The following result is a straightforward generalization of formulas~(\ref{3}) and~(\ref{4}):

\begin{prop}
\label{technical}
  There exist $\{b_i\}_{i=1}^{n}\in S(\g)^{\ad \g}[\bar{\rho}(\wt{F}_s),\ldots,\bar{\rho}(\wt{F}_m)]$ such that:
 $$\bar{\rho}(\wt{F}_{m+i})\equiv s_{n,m}\tau_i+b_i\mod \CC[\bar{\rho}(\wt{F}_s),\ldots,\bar{\rho}(\wt{F}_{m+i-1})],\ \ \ \forall\ 1\leq i\leq n, $$
 where $s_{n,m}=(-1)^m$ for $\g=\gl_n$ and $s_{n,m}=1/4$ for $\g=\spn_{2n}$.
\end{prop}

 Define $t_k\in H_m(\gl_n)$ by $t_k:=\sum_{i=1}^n x_i [Q_k,y_i]$ and $t_k\in H_m(\spn_{2n})$ by $t_k:=\sum_{i=1}^{2n}[Q_k,y_i]y_i^*$.
 Combining Proposition~\ref{technical}, Theorems~\ref{center},~\ref{main 1} with
 $\gr(Z(U(\g,e)))=\zz_{\mathrm{Pois}}(\CC[\s])$ we get

\begin{cor}\label{center_explicit}
 For $\g$ being either $\gl_n$ or $\spn_{2n}$,
there exist $C_1,\ldots,C_n\in Z(\U(\g))[\z_0,\ldots,\z_{m-s}]$,
such that the center $Z(H_m(\g))$ is a polynomial algebra in free
generators $\{\z_i\}\cup\{t_j+C_j\}_{j=1}^{n}$.
\end{cor}

 Considering the quotient of $H_m(\g)$ by the ideal $(\z_0-a_0,\ldots,\z_{m-s}-a_{m-s})$ for any $a_i\in \CC$,
we see that the center of the standard infinitesimal Cherednik
algebra $H_a(\g)$ contains a polynomial subalgebra
$\CC[t_1+c_1,\ldots,t_n+c_n]$ for some $c_j\in Z(\U(\g))$.

 Together with~\cite[Theorems 5.1 and 7.1]{DT} this yields:

\begin{cor}\label{tik_reprove}
 We actually have $Z(H_a(\g))=\CC[t_1+c_1,\ldots,t_n+c_n]$.
\end{cor}

 For $\g=\gl_n$ this is~\cite[Theorem 1.1]{T1}, while for $\g=\spn_{2n}$ this is~\cite[Conjecture 7.1]{DT}.

\subsection{Symplectic leaves of Poisson infinitesimal Cherednik algebras}\label{poisson_cher}
   By Theorem~\ref{main 2}, we get an identification of the full Poisson-central
 reductions of the algebras $\CC[\s_{n,m}]$ and $H_m^{\cl}(\gl_n)$ or $H_m^{\cl}(\spn_{2n})$.
 As an immediate consequence we obtain the following proposition, which answers a question raised in~\cite{DT}:
\begin{prop}
 Poisson varieties corresponding to arbitrary full central reductions of Poisson infinitesimal Cherednik algebras $H_\z^{\cl}(\g)$ have finitely many symplectic leaves.
\end{prop}

\subsection{Analogue of Kostant's theorem}

  As another immediate consequence of Theorem~\ref{main 1} and discussions from Section~\ref{centers}, we
  get a generalization of the following classical result:

\begin{prop}
 (a) The infinitesimal Cherednik algebras $H_\z(\g)$ are free over their
centers.

\noindent (b) The full central reductions of $\gr H_\z(\g)$ are
normal, complete intersection integral domains.
\end{prop}

 For $\g=\gl_n$ this is~\cite[Theorem 2.1]{T2}, while for $\g=\spn_{2n}$ this is~\cite[Theorem 8.1]{DT}.

\subsection{Category $\mathcal{O}$ and finite dimensional representations of $H_m(\spn_{2n})$}
The  categories $\mathcal{O}$ for the finite $W$-algebras were first
introduced in~\cite{BGK} and were further studied by the first
author in~\cite{L6}. Namely, recall that we have an embedding
$\mathfrak{q}\subset U(\g,e)$. Let $\mathfrak{t}$ be a Cartan
subalgebra of $\mathfrak{q}$ and set
$\g_0:=\mathfrak{z}_{\g}(\mathfrak{t})$. Pick an integral element
$\theta\in \mathfrak{t}$ such that $\mathfrak{z}_{\g}(\theta)=\g_0$.
By definition, the category $\mathcal{O}$ (for $\theta$) consists of
all finitely generated $U(\g,e)$-modules $M$, where the action of
$\mathfrak{t}$ is diagonalizable with finite dimensional eigenspaces
and, moreover, the set of weights is bounded from above in the sense
that there are complex numbers $\alpha_1,\ldots,\alpha_k$ such that
for any weight $\lambda$ of $M$ there is $i$ with
$\alpha_i-\langle\theta,\lambda\rangle\in \mathbb{Z}_{\leqslant 0}$.
The category $\mathcal{O}$ has analogues of Verma modules,
$\Delta(N^0)$. Here $N^0$ is an irreducible module over the
$W$-algebra $U(\g_0,e)$, where $\g_0$ is the centralizer of
$\mathfrak{t}$. In the cases of interest ($(\g,e)=(\ssl_{n+m},e_m),
(\spn_{2n+2m},e_m)$), we have $\g_0=\gl_m\times \mathbb{C}^{n-1},
\g_0=\spn_{2m}\times \mathbb{C}^n$ and $e$ is principal in $\g_0$.
In this case, the $W$-algebra $U(\g_0,e)$ coincides with the center
of $U(\g_0)$. Therefore $N^0$ is a one-dimensional space, and the
set of all possible $N^0$ is identified, via the Harish-Chandra
isomorphism, with the quotient $\h^*/W_0$, where $\h,W_0$ are a
Cartan subalgebra and the Weyl group of $\g_0$ (we take the quotient
with respect to the dot-action of $W_0$ on $\h^*$). As in the usual
BGG category $\mathcal{O}$, each Verma module has a unique
irreducible quotient, $L(N^0)$. Moreover, the map $N^0\mapsto
L(N^0)$ is a bijection between the set of finite dimensional
irreducible $U(\g_0,e)$-modules, $\h^*/W_0$, in our case, and the
set of irreducible objects in $\mathcal{O}$. We remark that all
finite dimensional irreducible modules lie in $\mathcal{O}$.

One can define a formal character for a module $M\in \mathcal{O}$. The characters of Verma
modules are easy to compute basically thanks to \cite[Theorem 4.5(1)]{BGK}. So to compute the characters
of the simples, one needs to determine the multiplicities of the simples in the Vermas.
This was done in \cite[Section 4]{L6} in the case when $e$ is principal in $\g_0$.
The multiplicities are given by values of certain Kazhdan-Lusztig polynomials at $1$
and so are hard to compute, in general. In particular, one cannot classify finite dimensional
irreducible modules just using those results.

When $\g=\ssl_{n+m}$, a classification of the finite dimensional
irreducible $U(\g,e)$-modules was obtained in \cite{BK2}; this
result is discussed in the next section. When $\g=\spn_{2n+2m}$, one
can describe the finite dimensional irreducible representations
using  \cite[Theorem 1.2.2]{L5}. Namely, the centralizer of $e$ in
$\operatorname{Ad}(\g)$ is  connected. So, according to \cite{L5},
the finite dimensional irreducible $U(\g,e)$-modules are in
one-to-one correspondence with the primitive ideals
$\mathcal{J}\subset U(\g)$ such that the associated variety of
$U(\g)/\mathcal{J}$ is $\overline{\mathbb{O}}$, where we write
$\mathbb{O}$ for the adjoint orbit of $e$. The set of such primitive
ideals is computable (for a fixed central character, those are in
one-to-one correspondence with certain left cells in the
corresponding integral Weyl group), but we will not need details on
that.

One can also describe all $N^0\in \h^*/W_0$ such that   $\dim
L(N^0)<\infty$ when $e$ is principal in $\g_0$. This is done in
\cite[5.1]{L7}. Namely,  choose a representative $\lambda\in \h^*$
of $N^0$ that is {\it antidominant} for $\g_0$ meaning that
$\langle\alpha^\vee,\lambda\rangle\not\in \mathbb{Z}_{>0}$ for any
positive root $\alpha$ of $\g_0$. Then we can consider the
irreducible highest weight module $L(\lambda)$ for $\g$ with highest
weight $\lambda-\rho$. Let $\mathcal{J}(\lambda)$ be its annihilator
in $U(\g)$, this is a primitive ideal that depends only on $N^0$ and
not on the choice of $\lambda$. Then $\dim L(N^0)<\infty$ if and
only if the associated variety of $U(\g)/\mathcal{J}(\lambda)$ is
$\overline{\mathbb{O}}$. The associated variety is computable thanks
to results of \cite{BV_class}; however this computation requires
quite a lot of combinatorics. It seems that one can still give a
closed combinatorial answer for $(\spn_{2n+2m},e_m)$ similar to that
for $(\ssl_{n+m},e_m)$ but we are not going to elaborate on that.

Now let us discuss the infinitesimal Cherednik algebras. In the
$\gl_n$-case the category $\mathcal{O}$ was defined
in~\cite[Definition 4.1]{T1} (see also~\cite[Section 5.2]{EGG}).
Under the isomorphism of Theorem~\ref{main 1}(a), that category
$\mathcal{O}$ basically coincides with its $W$-algebra counterpart.
The classification of finite dimensional irreducible modules and the
character computation in that case was done in \cite{DT}, but the
character formulas for more general simple modules were not known.
For the algebras $H_m(\spn_{2n})$, no category $\mathcal{O}$ was
introduced, in general; the case $n=1$ was discussed in~\cite{Kh}.
The classification of finite dimensional irreducible modules was not
known either.

\subsection{Finite dimensional representations of $H_m(\gl_n)$}
  Let us compare classifications of the finite dimensional
irreducible representations of $\U(\ssl_{n+m},e_m)$ from \cite{BK2}
and $H_a(\gl_n)$ from \cite{DT}.

 In the notation of~\cite{BK2}\footnote{\ In the \emph{loc.cit.}
$\g=\gl_{n+m}$, rather then $\ssl_{n+m}$. Nevertheless, it is not
very crucial since $\gl_{n+m}=\ssl_{n+m}\oplus \CC$.}, a nilpotent
element $e_m\in \gl_{n+m}$ corresponds to the partition
$(1,\ldots,1,m)$ of $n+m$. Let $S_m$ act on $\CC^{n+m}$ by permuting
the last $m$ coordinates. According to~\cite[Theorem 7.9]{BK2},
there is a bijection between the irreducible finite dimensional
representations of $\U(\gl_{n+m},e_m)$ and the orbits of the
$S_m$-action on $\CC^{n+m}$ containing a strictly dominant
representative. An element $\bar{\nu}=(\nu_1,\ldots,\nu_{n+m})\in
\CC^{n+m}$ is called strictly dominant if $\nu_i-\nu_{i+1}$ is a
positive integer for all $1\leq i\leq n$. The corresponding
irreducible $\U(\gl_{n+m},e_m)$-representation is denoted
$L_{\bar{\nu}}$.
  Viewed as a $\gl_n$-module (since $\gl_n=\q\subset \U(\gl_{n+m},e_m)$), $L_{\bar{\nu}}=L'_{\bar{\nu}}\oplus \bigoplus_{i\in I}  L'_{\eta_i}$,
 where $L'_{\eta}$ is the highest weight $\eta$ irreducible $\gl_n$-module, $\bar{\nu}:=(\nu_1,\ldots,\nu_n)$ and $I$ denotes some set of weights $\eta<\bar{\nu}$.

  Let us now recall~\cite[Theorem 4.1]{DT}, which classifies all irreducible finite dimensional representations of the
 infinitesimal Cherednik algebra $H_a(\gl_n)$. They turn out to be parameterized by strictly dominant $\gl_n$-weights
 $\bar{\lambda}=(\lambda_1,\ldots,\lambda_n)$ (that is $\lambda_i-\lambda_{i+1}$ is a positive integer for every $1\leq i<n$),
 for which there exists a positive integer $k$ satisfying $P(\bar{\lambda})=P(\lambda_1,\ldots,\lambda_{n-1},\lambda_n-k)$.
  Here $P$ is a degree $m+1$ polynomial function on the Cartan subalgebra $\h_n$ of all diagonal matrices of $\gl_n$, introduced in~\cite[Section 3.2]{DT}.
 According to~\cite[Theorem 3.2]{DT} (see Theorem~\ref{Casimir}(b) below), we have $P=\sum_{j\geq 0} w_jh_{j+1}$, where both $w_j$ and $h_j$
 are defined in the next section (see the notation preceding Theorem~\ref{Casimir}).

 These two descriptions are intertwined by a natural bijection,
sending $\bar{\nu}=(\nu_1,\ldots,\nu_{n+m})$ to
$\bar{\lambda}:=(\nu_1,\ldots,\nu_n)$, while
$\bar{\lambda}=(\lambda_1,\ldots,\lambda_n)$ is sent to the class of
$\bar{\nu}=(\lambda_1,\ldots,\lambda_n,\nu_{n+1},\ldots,\nu_{n+m})$
with $\{\nu_{n+1},\ldots,\nu_{n+m}\}\cup\{\lambda_n\}$ being the set
of roots of the polynomial
$P(\lambda_1,\ldots,\lambda_{n-1},t)-P(\bar{\lambda})$.

\subsection{Explicit isomorphism in the case $\g=\gl_n$}\label{Section Casimir}
  We compute the images of particular central elements of $H_m(\gl_n)$ and $U(\ssl_{n+m},e_m)$ under the corresponding Harish-Chandra isomorphisms.
 Comparison of these images enables us to determine the isomorphism $\bar{\Theta}$ of Theorem~\ref{main 1}(a) explicitly,
 in the same way as Theorem~\ref{main 2}(a) was deduced.

 Let us start from the following commutative diagram:

\medskip

\setlength{\unitlength}{1cm}
\begin{picture}(4,3)
 \put(4.6,0.5){$U(\gl_n)\otimes U(\ssl_m,e_m)$}
 \put(5,2.5){$U(\ssl_{n+m},e_m)_0$}
 \put(9.5,0.5){$Z(U(\gl_n))\otimes U(\ssl_m,e_m)$}
 \put(9.5,2.5){$Z(U(\ssl_{n+m},e_m))$}
 \put(0,1.5){$U(\ssl_{n+m},e_m)^0$}

 \put(5.8,2.3){\vector (0,-1){1.5}}
 \put(10.5,2.3){\vector (0,-1){1.5}}
 \put(9.3,2.6){\vector (-1,0){2}}
 \put(9.3,0.6){\vector (-1,0){1.5}}
 \put(4.7,2.3){\vector (-4,-1){2.4}}
 \put(2.3,1.3){\vector (4,-1){2.2}}

 \put(8,2.8){$j_{n,m}$}
 \put(8.3,0.8){$j_{n}\otimes \mathrm{Id}$}
 \put(5.4,1.5){$\varpi$}
 \put(9.9,1.5){$\varphi^W$}
 \put(3.3,2.1){$\pi$}
 \put(3.3,1.1){$o$}

\end{picture}
 (Diagram 1)

\medskip
 In the above diagram:

\noindent
 $\bullet$ $U(\ssl_{n+m},e_m)_0$ is the $0$-weight component of $U(\ssl_{n+m},e_m)$ with respect to the grading
 $\Gr$.

\noindent
 $\bullet$ $U(\ssl_{n+m},e_m)^0:=U(\ssl_{n+m},e_m)_0/ (U(\ssl_{n+m},e_m)_0\cap
 U(\ssl_{n+m},e_m)U(\ssl_{n+m},e_m)_{>0})$.

\noindent
 $\bullet$ $\pi$ is the quotient map, while $o$ is an isomorphism, constructed in~\cite[Theorem 4.1]{L6}.
   \footnote{\ Here we actually use the fact that $U(\gl_n)\otimes U(\ssl_m,e_m)$ is the finite $W$-algebra $U(\gl_n\oplus \ssl_m, 0\oplus e_m)$.}

\noindent
 $\bullet$ The homomorphism $\varpi$ is defined as $\varpi:=o\circ \pi$, making the triangle
 commutative.

\noindent
 $\bullet$ The homomorphisms $j_{n+m},\ j_n$ are the natural inclusions.

\noindent
 $\bullet$ The homomorphism $\varphi^W$ is the restriction of $\varpi$ to the center, making the square
 commutative.

\noindent
 $\bullet$ $U(\ssl_m,e_m)\cong Z(U(\ssl_m,e_m))\cong Z(U(\ssl_m))$ since $e_m$ is a principal nilpotent of
 $\ssl_m$.

\medskip
 We have an analogous diagram for the universal infinitesimal Cherednik algebra of $\gl_n$:

\medskip

\setlength{\unitlength}{1cm}
\begin{picture}(4,3.2)
 \put(4,0.5){$U(\gl_n)\otimes \CC[\z_0,\ldots,\z_{m-2}]$}
 \put(4.5,2.5){$H_m(\gl_n)_0$}
 \put(9.5,0.5){$Z(U(\gl_n))\otimes \CC[\z_0,\ldots,\z_{m-2}]$}
 \put(10.4,2.5){$Z(H_m(\gl_n))$}
 \put(0,1.5){$H_m(\gl_n)^0$}

 \put(5,2.3){\vector (0,-1){1.4}}
 \put(11,2.3){\vector (0,-1){1.4}}
 \put(10.1,2.6){\vector (-1,0){4}}
 \put(9,0.6){\vector (-1,0){1}}
 \put(4.5,2.3){\vector (-4,-1){2.9}}
 \put(1.6,1.4){\vector (4,-1){2.3}}

 \put(7.8,2.8){$j'_{n,m}$}
 \put(8.1,0.8){$j_{n}\otimes \mathrm{Id}$}
 \put(4.5,1.5){$\varpi'$}
 \put(10.5,1.5){$\varphi^H$}
 \put(3,2.1){$\pi'$}
 \put(3,1.1){$o'$}
\end{picture}
 (Diagram 2)

\medskip
In the above diagram:

\noindent
 $\bullet$ $H_m(\gl_n)_0$ is the degree $0$ component of $H_m(\gl_n)$ with respect to the grading $\Gr$, defined by setting
  $\deg(\gl_n)=\deg(\z_0)=\ldots=\deg(\z_{m-2})=0,\ \deg(V_n)=1,\
  \deg(V_n^*)=-1$.

\noindent
 $\bullet$ $H_m(\gl_n)^0$ is the quotient of $H_m(\gl_n)_0$ by $H_m(\gl_n)_0\cap
 H_m(\gl_n)H_m(\gl_n)_{>0}$.
           \footnote{\ It is easy to see that $H_m(\gl_n)_0\cap H_m(\gl_n)H_m(\gl_n)_{>0}$ is actually a two-sided ideal of $H_m(\gl_n)_0$.}

\noindent
 $\bullet$ $\pi'$ denotes the quotient map, $o'$ is the natural
 isomorphism, $\varpi':=o'\circ \pi'$.

\noindent
 $\bullet$ The inclusion $j'_{n,m}$ is a natural inclusion of the
 center.

\noindent
 $\bullet$ The homomorphism $\varphi^H$ is the one induced by restricting $\varpi'$ to the
 center.

\medskip

  The isomorphism $\bar{\Theta}$ of Theorem~\ref{main 1}(a) intertwines the gradings $\Gr$,
 inducing an isomorphism $\bar{\Theta}^0:H_m(\gl_n)^0\iso U(\ssl_{n+m},e_m)^0$.
 This provides the following commutative diagram:

\setlength{\unitlength}{1cm}
\begin{picture}(4,3.2)
 \put(0.5,0.5){$Z(U(\gl_n))\otimes \CC[\z_0,\ldots,\z_{m-2}]$}
 \put(1,2.5){$Z(H_m(\gl_n))$}
 \put(9,0.5){$Z(U(\gl_n))\otimes Z(U(\ssl_m))$}
 \put(9.5,2.5){$Z(U(\ssl_{n+m},e_m))$}

 \put(2,2.3){\vector (0,-1){1.4}}
 \put(10.4,2.3){\vector (0,-1){1.4}}

 \put(2.9,2.6){\vector (1,0){6.2}}
 \put(5.4,0.6){\vector (1,0){3.2}}

 \put(6.1,2.8){$\vartheta$}
 \put(7,0.8){$\underline{\vartheta}$}
 \put(1.5,1.5){$\varphi^H$}
 \put(9.8,1.5){$\varphi^W$}
\end{picture}
 (Diagram 3)

\medskip
 In the above diagram:

\noindent
 $\bullet$ The isomorphism $\vartheta$ is the restriction of the isomorphism $\bar{\Theta}$ to the center.

\noindent
 $\bullet$ The isomorphism $\underline{\vartheta}$ is the restriction of the isomorphism $\bar{\Theta}^0$ to the center.

\medskip
  Let $\HC_N$ denote the Harish-Chandra isomorphism $\HC_N:Z(U(\gl_N))\iso \CC[\h_N^*]^{S_N,\bullet}$,
 where $\mathfrak{h}_N\subset \gl_N$ is the Cartan subalgebra consisting of the diagonal matrices and $(S_N,\bullet)$-action arises from the $\rho_N$-\emph{shifted}
 $S_N$-action on $\h_N^*$ with $\rho_N=(\frac{N-1}{2},\frac{N-3}{2},\cdots,\frac{1-N}{2})\in \mathfrak{h}_N^*$.
 This isomorphism has the following property: any central element $z\in Z(U(\gl_N))$ acts on the Verma module $M_{\lambda-\rho_N}$ of $U(\gl_N)$ via $\HC_N(z)(\lambda)$.

  According to Corollary~\ref{center_explicit}, the center $Z(H_m(\gl_n))$ is the polynomial algebra in free generators
 $\{\z_0,\ldots,\z_{m-2},t_1',\ldots,t_n'\}$, where $t_k'=t_k+C_k$.
 In particular, any central element of Kazhdan degree $2(m+1)$ has the form $ct_1'+p(\z_0,\ldots,\z_{m-2})$ for some $c\in \CC$ and $p\in \CC[\z_0,\ldots,\z_{m-2}]$.

  Following~\cite{DT}, we call $t_1'=t_1+C_1$ the Casimir element. \footnote{\ The Casimir element is uniquely defined up to a constant.}
 An explicit formula for $\varphi^H(t_1')$ is provided by~\cite[Theorem 3.1]{DT}, while for any $0\leq k\leq m-2$ we have $\varphi^H(\z_k)=1\otimes \z_k$.

  To formulate main results about the Casimir element $t'_1$, we introduce:

\noindent
 $\bullet$
  the generating series $\z(z)=\sum_{i=0}^{m-2}\z_iz^i+z^m$ (already introduced in Section~\ref{Poisson counterparts}),

\noindent
 $\bullet$
  a unique degree $m+1$ polynomial $f(z)$ satisfying $f(z)-f(z-1)=\partial^n(z^n\z(z))$ and $f(0)=0$,

\noindent
 $\bullet$
  a unique degree $m+1$ polynomial $g(z)=\sum_{i=1}^{m+1}g_iz^i$ satisfying $\partial^{n-1}(z^{n-1}g(z))=f(z)$,

\noindent
 $\bullet$
  a unique degree $m$ polynomial $w(z)=\sum_{i=0}^{m} w_iz^i$ satisfying $f(z)=(2\sinh (\partial/2))^{n-1} (z^n w(z))$,

\noindent
 $\bullet$
  the symmetric polynomials $\sigma_i(\lambda_1,\ldots,\lambda_n)$ via $(u+\lambda_1)\cdots(u+\lambda_n)=\sum \sigma_i(\lambda_1,\ldots,\lambda_n)u^{n-i}$,

\noindent
 $\bullet$
  the symmetric polynomials $h_j(\lambda_1,\ldots,\lambda_n)$ via $(1-u\lambda_1)^{-1}\cdots(1-u\lambda_n)^{-1}=\sum{h_j(\lambda_1,\ldots,\lambda_n)u^j}$,

\noindent
 $\bullet$
  the central element $H_j\in Z(U(\gl_n))$ which is the symmetrization of $\tr S^j (\cdot)\in \CC[\gl_n]\cong S(\gl_n)$.

\medskip
 The following theorem summarizes the main results of~\cite[Section 3]{DT}:

\begin{thm}\label{Casimir}

 (a) \cite[Theorem 3.1]{DT} $\varphi^H(t_1')=\sum_{j=1}^{m+1} H_j\otimes
 g_j$ (where $g_j$ are viewed as elements of $\CC[\z_0,\ldots,\z_{m-2}]$),

 (b) \cite[Theorem 3.2]{DT} $(\HC_n\otimes \mathrm{Id})\circ \varphi^H (t_1')=\sum_{j=0}^{m}{h_{j+1}\otimes w_j}$.
\end{thm}

\medskip
  Let $\HC'_N$ denote the Harish-Chandra isomorphism $Z(U(\ssl_N))\iso\CC[\bar{\h}_N^*]^{S_N,\bullet}$, where $\bar{\h}_N$ is the Cartan
 subalgebra of $\ssl_N$, consisting of the diagonal matrices, which can be identified with $\{(z_1,\ldots,z_N)\in \CC^N|\sum z_i=0\}$.
 The natural inclusion $\bar{\h}_N\hookrightarrow \h_N$ induces the map
 $$\h_N^*\to \bar{\h}_N^*:\ (\lambda_1,\ldots,\lambda_N)\mapsto (\lambda_1-\mu,\ldots,\lambda_N-\mu),\ \textit{where}\ \ \mu:=\frac{\lambda_1+\ldots+\lambda_N}{N}.$$

  The isomorphisms $\HC'_{n+m},\HC'_m,\HC_n$ fit into the following commutative diagram:

\setlength{\unitlength}{1cm}
\begin{picture}(4,3.2)
 \put(4.5,2.5){$Z(U(\ssl_{n+m}))$}
 \put(4,0.5){$Z(U(\gl_n))\otimes Z(U(\ssl_m))$}
 \put(10.5,2.5){$\CC[\CC^{n+m-1}]^{S_{n+m},\bullet}$}
 \put(10,0.5){$\CC[\CC^n]^{S_n,\bullet}\otimes \CC[\CC^{m-1}]^{S_m,\bullet}$}
 \put(-0.5,1.5){$Z(U(\ssl_{n+m},e_m))$}

 \put(5.3,2.3){\vector (0,-1){1.4}}
 \put(11.5,2.3){\vector (0,-1){1.4}}
 \put(4.3,2.3){\vector (-3,-1){2}}
 \put(2.3,1.4){\vector (3,-1){1.6}}

 \put(6.7,2.6){\vector (1,0){3.3}}
 \put(7.8,0.6){\vector (1,0){1.9}}

  \put(7.8,2.8){$\HC'_{n+m}$}
  \put(7.9,0.8){$\HC_n\otimes \HC'_m$}
  \put(4.6,1.5){$\bar{\varphi}^W$}
  \put(10.9,1.5){$\varphi^C$}
  \put(3.1,2.1){$\rho$}
  \put(2.8,0.7){$\varphi^W$}

\end{picture}
 (Diagram 4)

\medskip
 In the above diagram:

\noindent
 $\bullet$ $\rho$ is the isomorphism of Theorem~\ref{center}.

\noindent
 $\bullet$ The homomorphism $\bar{\varphi}^W$ is defined as the composition $\bar{\varphi}^W:=\varphi^W\circ \rho$.

\noindent
 $\bullet$ The homomorphism $\varphi^C$ arises from an identification $\CC^n\times \CC^{m-1}\cong \CC^{n+m-1}$ defined by
 $$(\lambda_1,\ldots,\lambda_n,\nu_1,\ldots,\nu_m)\mapsto \left(\lambda_1,\ldots,\lambda_n,\nu_1-\frac{\lambda_1+\ldots+\lambda_n}{m},
    \ldots,\nu_m-\frac{\lambda_1+\ldots+\lambda_n}{m}\right).$$

 In particular, $\varphi^C$ is injective, so that $\varphi^W$ is injective and, hence, $\varphi^H$ is injective.

\medskip
 Define $\bar{\sigma}_k\in \CC[\bar{\h}_N^*]$ as the restriction of $\sigma_k$ to $\CC^{N-1}\hookrightarrow
 \CC^N$. According to Lemma~\ref{block 2},
\begin{equation}\label{added}
 \varphi^C(\bar{\sigma}_{m+1})=(-1)^mh_{m+1}\otimes 1+\sum_{j=2}^m (-1)^{m-j}h_{m+1-j}\otimes 1\cdot \varphi^C(\bar{\sigma}_j).
\end{equation}

  Define $S_k\in Z(U(\ssl_{n+m}))$ by $S_k:=(\HC'_{n+m})^{-1}(\bar{\sigma}_k)$ for all $0\leq k\leq n+m$, so that $S_0=1,\ S_1=0$.
 Similarly, define $T_k\in Z(U(\gl_n))$ as $T_k:=\HC_n^{-1}(h_k)$ for all $k\geq 0$, so that $T_0=1$.

 Equality~(\ref{added}) together with the commutativity of Diagram 4 imply
  $$\bar{\varphi}^W(S_{m+1})=(-1)^mT_{m+1}\otimes 1+ \sum_{j=2}^m (-1)^{m-j}T_{m+1-j}\otimes 1\cdot  \bar{\varphi}^W(S_j).$$

  According to our proof of Theorem~\ref{main 1}(a),
 we have $\bar{\Theta}(A)=\Theta(A)+s \tr A$ for all $A\in \gl_n$, where $s=-\frac{\eta_{m-1}}{(n+m)\eta_m}$.
  In particular, $\underline{\vartheta}^{-1}(X\otimes 1)=\varphi_{-s}(X)\otimes 1$ for all $X\in Z(U(\gl_n))$, where
 $\varphi_{-s}$ was defined in Lemma~\ref{less_parameters}.

  As a consequence, we get:
\begin{equation}\label{hren}
 \underline{\vartheta}^{-1}(\bar{\varphi}^W(S_{m+1}))=(-1)^m\varphi_{-s}(T_{m+1})\otimes 1+
   \sum_{j=2}^m (-1)^{m-j}\varphi_{-s}(T_{m+1-j})\otimes 1\cdot \underline{\vartheta}^{-1}(\bar{\varphi}^W(S_j)).
\end{equation}

 The following identity is straightforward:
\begin{lem}\label{twist}
 For any positive integer $i$ and any constant $\delta\in \CC$ we have
  $$h_i(\lambda_1+\delta,\ldots,\lambda_n+\delta)=\sum_{j=0}^i \binom{n+i-1}{j}h_{i-j}(\lambda_1,\ldots,\lambda_n)\delta^j.$$
\end{lem}

 As a result, we get
\begin{equation}\label{hren2}
 \varphi_{-s}(T_i)=\sum_{j=0}^i \binom{n+i-1}{j}(-s)^jT_{i-j}.
\end{equation}

 Combining equations~(\ref{hren}) and~(\ref{hren2}), we get:
\begin{equation}\label{hren3}
 \underline{\vartheta}^{-1}(\bar{\varphi}^W(S_{m+1}))=(-1)^mT_{m+1}\otimes
 1+ (-1)^{m+1}s(n+m)T_m\otimes 1+ \sum_{l=-1}^{m-2} (-1)^lT_{l+1}\otimes 1\cdot  \bar{V}_l,
\end{equation}
 where $\bar{V}_l=\underline{\vartheta}^{-1}(\bar{\varphi}^W(V_l))$ and for $0\leq l\leq m-2$ we have
$$  V_l=\sum_{0\leq j\leq m-l} s^{m-l-j}\binom{n+m-j}{m-l-j}S_j.$$

  On the other hand, the commutativity of Diagram 3 implies
 $$\underline{\vartheta}^{-1}(\bar{\varphi}^W(S_{m+1}))=\varphi^H(\vartheta^{-1}(\rho(S_{m+1}))).$$

 Recall that there exist $c\in \CC,\ p\in \CC[\z_0,\ldots\z_{m-2}]$ such that $\vartheta^{-1}(\rho(S_{m+1}))=ct'_1+p$.
 As $\varphi^H(\z_i)=1\otimes \z_i$ and $\varphi^H(t_1')=\sum_{j=0}^m T_{j+1}\otimes w_j$ (by Theorem~\ref{Casimir}(b)), we get
\begin{equation}\label{non-hren}
 \varphi^H(\vartheta^{-1}(\rho(S_{m+1})))=1\otimes p(\z_0,\ldots,\z_{m-2})+\sum_{0\leq j\leq m}T_{j+1}\otimes cw_j.
\end{equation}

 Since $w_m=1, w_{m-1}=\frac{n+m}{2}$, the comparison of~(\ref{hren3}) and~(\ref{non-hren}) yields:

\noindent
 $\bullet$ The coefficients of $T_{m+1}$ must coincide, so that $(-1)^m=cw_m\Rightarrow c=(-1)^m$.

\noindent
 $\bullet$ The coefficients of $T_m$ must coincide, so that $cw_{m-1}=(-1)^{m+1}(n+m)s\Rightarrow s=-1/2$.

\noindent
 $\bullet$ The coefficients of $T_{j+1}$ must coincide for all $j\geq 0$, so that
              $$w_j=(-1)^{m-j}\bar{V}_j\Rightarrow\vartheta(w_j)=(-1)^{m-j}\rho(V_j).$$

\medskip
  Recall that $\bar{\eta}_m=1$, and so $\eta_m=\bar{\eta}_m=1$.
 As a result $s=-\frac{\eta_{m-1}}{n+m}$, so that $\eta_{m-1}=\frac{n+m}{2}$.

\medskip
\noindent
 The above discussion can be summarized as follows:
\begin{thm}\label{main_explicit}
  Let $\bar{\Theta}:H_m(\gl_n)\iso U(\ssl_{n+m},e_m)$ be the isomorphism from Theorem~\ref{main 1}(a).
 Then $\bar{\Theta}(A)=\Theta(A)-\frac{1}{2}\tr A,\ \bar{\Theta}(y)=\Theta(y),\ \bar{\Theta}(x)=\Theta(x),$
 while $\bar{\Theta}\mid_{\CC[\z_0,\ldots,\z_{m-2}]}$ is uniquely determined by
 $\bar{\Theta}(w_j)=(-1)^{m-j}\rho(V_j)$ for all $0\leq j\leq m-2$.
\end{thm}

\subsection{Higher central elements}

   It was conjectured in~\cite[Remark 6.1]{DT}, that the action of central elements $t_i'=t_i+c_i\in Z(H_m(\gl_n))$ on the Verma modules of $H_a(\gl_n)$
 should be obtained from the corresponding formulas at the the Poisson level (see Theorem~\ref{Poisson-center})
 via a \emph{basis change} $\z(z)\rightsquigarrow w(z)$ and a \emph{$\rho_n$-shift}.
  Actually, that is not true. However, we can choose another set of generators $u_i\in Z(H_m(\gl_n))$, whose action is given by
 formulas similar to those of Theorem~\ref{Poisson-center}.

 Let us define:

\noindent
 $\bullet$ central elements $u_i\in Z(H_m(\gl_n))$ by $u_i:=\underline{\vartheta}^{-1}(\rho(S_{m+i}))$ for all $0\leq i\leq n$,

\noindent
 $\bullet$ the generating polynomial $\wt{u}(t):=\sum_{i=0}^n (-1)^iu_it^i$,

\noindent
 $\bullet$ the generating polynomial $S(z):=\sum_{i=0}^n (-1)^i\underline{\vartheta}^{-1}(\bar{\varphi}^W(S_{m-i}))z^i\in \CC[\z_0,\ldots,\z_{m-2};z]$.

\medskip
 The following result is proved using the arguments of the Section~\ref{Section Casimir}:

\begin{thm}
 We have:
 $$(\HC_n\otimes \mathrm{Id})\circ \varphi^H(\wt{u}(t))=
   (\varphi_{1/2}\otimes \mathrm{Id})\left(\mathrm{Res}_{z=0}\ {S(z^{-1})\prod_{1\leq i\leq n}{\frac{1-t\lambda_i}{1-z\lambda_i}}\frac{z^{-1}dz}{1-t^{-1}z}}\right).$$
\end{thm}


$\ $

\section{Completions}

\subsection{Completions of graded deformations of Poisson
algebras} We first recall the machinery of completions, elaborated
by the first author (our exposition follows~\cite{L2}). Let $Y$ be
an affine Poisson scheme equipped with a $\CC^*$-action, such that
the Poisson bracket has degree $-2$. Let $\Aa_\hbar$ be an
associative flat graded $\CC[\hbar]$-algebra (where $\deg(\hbar)=1$)
such that $[\Aa_\hbar,\Aa_\hbar]\subset \hbar^2\Aa_\hbar$ and
$\CC[Y]=\Aa_\hbar/(\hbar)$ as a graded Poisson algebra. Pick a point
$x\in Y$ and let $I_x\subset \CC[Y]$ be the maximal ideal of $x$,
while $\wt{I}_x$ will denote its inverse image in $\Aa_\hbar$.

\medskip
\begin{defn}
 The completion of $\Aa_\hbar$ at $x\in Y$ is by definition $\Aa_\hbar^{\wedge_x}:=\underset{\longleftarrow}\lim\ \Aa_\hbar/\wt{I}_x^n$.
\end{defn}
\medskip

 This is a complete topological $\CC[[\hbar]]$-algebra, flat over $\CC[[\hbar]]$,
such that $\Aa_\hbar^{\wedge_x}/(\hbar)=\CC[Y]^{\wedge_x}$.
 Our main motivation for considering this construction is the decomposition theorem, generalizing the corresponding classical result at the Poisson
level:

\begin{prop}\cite[Theorem 2.3]{K}
 The formal completion $\widehat{Y}_x$ of $Y$ at $x\in Y$ admits a
 product decomposition $\widehat{Y}_x=\mathcal{Z}_x\times
 \widehat{Y}_{x}^s$, where $Y^s$ is the symplectic leaf of $Y$
 containing $x$ and $\mathcal{Z}_x$ is a local formal Poisson
 scheme.
\end{prop}

 Fix a maximal symplectic subspace $V\subset T_x^*Y$. One can choose an embedding
$V\overset{i}\hookrightarrow\wt{I}^{\wedge_x}_x$ such that
$[i(u),i(v)]=\hbar^2\omega(u,v)$ and composition
$V\overset{i}\hookrightarrow\wt{I}_x^{\wedge_x}\twoheadrightarrow
T_x^*Y$ is the identity map. Finally, we define
$W_\hbar(V):=T(V)[h]/(u\otimes v-v\otimes u-\hbar^2\omega(u,v))$,
which is graded by setting $\deg(V)=1,\ \deg(\hbar)=1$ (the
\emph{homogenized Weyl algebra}). Then we have:

\begin{thm}\cite[Sect. 2.1]{L2}[Decomposition theorem]\label{decomposition}
 There is a splitting $\Aa_\hbar^{\wedge_x}\cong W_\hbar(V)^{\wedge_0}{\widehat{\otimes}}_{\CC[[\hbar]]}\underline{\Aa}'_\hbar$,
 where $\underline{\Aa}'_\hbar$ is the centralizer of $V$ in $\Aa_\hbar^{\wedge_x}$.
\end{thm}

\begin{rem}
 Recall that a filtered algebra $\{F_i(B)\}_{i\geq 0}$ is called a \emph{filtered deformation} of $Y$ if $\gr_{F_\bullet} B\cong \CC[Y]$ as Poisson graded algebras.
 Given such $B$, we set $\Aa_\hbar:=\mathrm{Rees}_\hbar(B)$ (the Rees algebra of the filtered algebra $B$), which
 naturally satisfies all the above conditions.
 \end{rem}

\noindent
 This remark provides the following interesting examples of $\Aa_\hbar$:

\medskip \noindent
$\bullet$  \emph{The homogenized Weyl algebra.}

 Algebra $W_\hbar(V)$ from above is obtained via the Rees construction from the usual Weyl algebra.
 In the case $V=V_n\oplus V_n^*$ with a natural symplectic form, we denote $W_\hbar(V)$ just by $W_{\hbar,n}$.

\medskip \noindent
 $\bullet$ \emph{The homogenized universal enveloping algebra.}

 For any graded Lie algebra $\g=\bigoplus \g_i$ with a Lie bracket of degree $-2$, we define
  $$\U_\hbar(\g):=T(\g)[\hbar]/(x\otimes y-y\otimes x-\hbar^2[x,y]|x,y\in \g),$$
 graded by setting $\deg(\g_i)=i,\ \deg(\hbar)=1.$

\medskip \noindent
 $\bullet$ \emph{The homogenized universal infinitesimal Cherednik algebra of $\gl_n$.}

  Define $H_{\hbar,m}(\gl_n)$ as a quotient $H_{\hbar,m}(\gl_n):=\U_\hbar(\gl_n)\ltimes T(V_n\oplus V_n^*)[\z_0,\ldots,\z_{m-2}]/J$, where
    $$J=\left([x,x'], [y,y'], [A,x]-\hbar^2A(x), [A,y]-\hbar^2A(y), [y,x]-\hbar^2(\sum_{j=0}^{m-2}\z_j r_j(y,x)+r_m(y,x))\right).$$
 This algebra is graded by setting $\deg(V_n\oplus V_n^*)=m+1,\ \deg(\z_i)=2(m-i)$.

\medskip \noindent
 $\bullet$ \emph{The homogenized universal infinitesimal Cherednik algebra of $\spn_{2n}$.}

  Define $H_{\hbar,m}(\spn_{2n})$ as a quotient $H_{\hbar,m}(\spn_{2n}):=\U_\hbar(\spn_{2n})\ltimes T(V_{2n})[\z_0,\ldots,\z_{m-1}]/J$, where
    $$J=\left([A,y]-\hbar^2A(y), [x,y]-\hbar^2(\sum_{j=0}^{m-1}\z_jr_{2j}(x,y)+r_{2m}(x,y))|A\in \spn_{2n},\ x,y\in V_{2n}\right).$$
 This algebra is graded by setting $\deg(V_{2n})=2m+1,\ \deg(\z_i)=4(m-i)$.

\medskip \noindent
 $\bullet$ \emph{The homogenized $W$-algebra.}

   The homogenized $W$-algebra, associated to $(\g,e)$ is defined by $U_\hbar(\g,e):=(\U_\hbar(\g)/\U_\hbar(\g)\m')^{\ad \m}$.

\medskip

 There are many interesting contexts in which Theorem~\ref{decomposition} proved to be a useful tool.
Among such let us mention Rational Cherednik algebras (\cite{BE}),
Symplectic Reflection algebras (\cite{L4}) and $W$-algebras
(\cite{L1,L2}).

\medskip
  Actually, combining results of~\cite{L2} with Theorem~\ref{main 1}, we get isomorphisms
\begin{equation} \tag{*}
  \Psi_m:H_{\hbar,m}(\gl_n)^{\wedge_v}\iso
  H_{\hbar,m+1}(\gl_{n-1})^{\wedge_0}{\widehat{\otimes}}_{\CC[[\hbar]]}
  W_{\hbar,n}^{\wedge_v},
\end{equation}
 \begin{equation} \tag{$\spadesuit$}
  \Upsilon_m:H_{\hbar,m}(\spn_{2n})^{\wedge_{v}}\iso
  H_{\hbar,m+1}(\spn_{2n-2})^{\wedge_0}{\widehat{\otimes}}_{\CC[[\hbar]]}
  W_{\hbar,2n}^{\wedge_{v}},
\end{equation}
 where $v\in V_n$ (respectively $v\in V_{2n}$) is a nonzero element and $m\geq 1$.

 These decompositions can be viewed as \emph{quantizations} of their Poisson versions:
\begin{equation} \tag{$\star$}
  \Psi_m^{\mathrm{cl}}:H_{m}^{\mathrm{cl}}(\gl_n)^{\wedge_{v}}\iso
  H_{m+1}^{\mathrm{cl}}(\gl_{n-1})^{\wedge_0}{\widehat{\otimes}}_{\CC} W_{n}^{\mathrm{cl},\wedge_{v}},
\end{equation}
\begin{equation} \tag{$\heartsuit$}
  \Upsilon_m^{\mathrm{cl}}:H_{m}^{\mathrm{cl}}(\spn_{2n})^{\wedge_{v}}\iso
  H_{m+1}^{\mathrm{cl}}(\spn_{2n-2})^{\wedge_0}{\widehat{\otimes}}_{\CC}
  W_{2n}^{\mathrm{cl},\wedge_{v}},
\end{equation}
 where $W_{n}^{\mathrm{cl}}\simeq
 \CC[x_1,\ldots,x_n,y_1,\ldots,y_n]$ with $\{x_i,x_j\}=\{y_i,y_j\}=0,
 \{x_i,y_j\}=\delta_i^j$.

  Isomorphisms (*) and ($\spadesuit$) are not unique and, what is worse, are  inexplicit.
%
%
%

\medskip
 Let us point out that localizing at other points of $\gl_n\times V_n\times V_n^*$ (respectively $\spn_{2n}\times V_{2n}$) yields other decomposition isomorphisms.
 In particular, one gets~\cite[Theorem 3.1]{T3}
 \footnote{\ This result is stated in~\cite{T3}. However, its proof in the \emph{loc. cit.} is computationally wrong.} as follows:

 \begin{rem}
 For $n=1,m>0,$ consider $e':=e_m+E_{1,2n+2}\in \s_{1,m}\subset \spn_{2m+2}$,
 which is a subregular nilpotent element of $\spn_{2m+2}$. Above arguments yield a decomposition isomorphism
 \begin{equation}\tag{$\clubsuit$}
   H_{\hbar,m}(\spn_2)^{\wedge_{E_{12}}}\iso
  U_\hbar(\spn_{2m+2},e')^{\wedge_0}{\widehat{\otimes}}_{\CC[[\hbar]]}
  W_{\hbar,1}^{\wedge_0}.
 \end{equation}
  The full central reduction of ($\clubsuit$) provides an isomorphism of~\cite[Theorem 3.1]{T3}.
  \footnote{\ To be precise, we use an isomorphism of the $W$-algebra $U(\spn_{2m+2},e')$ and the non-commutative deformation of Crawley-Boevey and Holland
             of type $D_{m+2}$ Kleinian singularity.}
\end{rem}

\medskip

 In Appendix C, we establish explicitly suitably modified versions of (*) and ($\spadesuit$) for the
 cases $m=-1,\ 0$, which do not follow from the above arguments. In particular, the reader will get a flavor of what the formulas look like.

$\ $


\newpage

\appendix

\section{Proof of Lemmas~\ref{less_parameters},~\ref{less_parameters2}}

\medskip
$\bullet$ \emph{Proof of Lemma~\ref{less_parameters}(a)}
\medskip

 Let $\phi:H_{\z}(\gl_n)\overset{\sim}\longrightarrow
H_{\z'}(\gl_n)$ be a filtration preserving isomorphism. We have
$\phi(1)=1$, so that $\phi$ is the identity on the $0$-th level of
the filtration.

 Since $\F^{(N)}_2(H_{\z}(\mathfrak{gl}_n))=\F^{(N)}_2(H_{\z'}(\mathfrak{gl}_n))=\U(\gl_n)_{\leq 1}$,
we have $\phi(A)=\psi(A)+\gamma(A),\ \forall A\in \gl_n$, with
$\psi(A)\in \gl_n, \gamma(A)\in \CC$. Then
$\phi([A,B])=[\phi(A),\phi(B)],\ \forall A,B\in \gl_n$, if and only
if $\gamma([A,B])=0$ and $\psi$ is an automorphism of the Lie
algebra $\gl_n$. Since $[\gl_n,\gl_n]=\ssl_n$, we have
$\gamma(A)=\lambda\cdot \tr A$ for some $\lambda\in\CC$. For $n\geq
3$, $\Aut(\gl_n)=\Aut(\ssl_n)\times \Aut (\CC)=(\mu_2\ltimes
\mathrm{SL}(n))\times \CC^*$, where $-1\in \mu_2$ acts on $\ssl_n$
via $\sigma:A\mapsto -A^t$. This determines $\phi$ up to the
filtration level $N-1$.

 Finally, $\F^{(N)}_N(H_{\z}(\mathfrak{gl}_n))=\F^{(N)}_N(H_{\z'}(\mathfrak{gl}_n))=V_n\oplus V_n^*\oplus \U(\gl_n)_{\leq
 N}$. As explained, $\phi_{|{\U(\gl_n)}}$ is parameterized by $(\epsilon, T, \nu, \lambda)\in (\mu_2\ltimes
\mathrm{SL}(n))\times \CC^*\times \CC$ (no $\mu_2$-factor for $n=1,
2$). Let $I_n\in \gl_n$ be the identity matrix. Note that $[I_n,
y]=y, [I_n, x]=-x, [I_n, A]=0$ for any $y\in V_n, x\in V_n^*, A\in
\gl_n$.
  Since $\phi(y)=[\nu\cdot
I_n+n\lambda,\phi(y)]=\nu[I_n,\phi(y)],\ \forall y\in V_n$, we get
$\nu=\pm 1$.

 \textbf{Case 1: $\nu=1$.} Then $\phi(y)\in V_n,\ \phi(x)\in V_n^*\ (\forall y\in
V_n,x\in V_n^*)$. Since $V_n\ncong V_n^\sigma$ as $\ssl_n$-modules
for $n\geq 3$ and $\End_{\ssl_n}(V_n)=\CC^*$, we get $\epsilon=1\in
\mu_2$ (so that $\phi(A)=TAT^{-1},\ \forall A\in \ssl_n$) and there
exist $\theta_1,\theta_2\in \CC^*$ such that $\phi(y)=\theta_1\cdot
T(y),\ \phi(x)=\theta_2\cdot T(x)\ (\forall y\in V_n, x\in V_n^*)$.
 Hence, we get $\varphi(T,\lambda)(\z(y,x))=\phi([y,x])=[\phi(y),\phi(x)]=\theta
\z'(T(y),T(x))$, where $\theta=\theta_1\theta_2$ and isomorphism
$\varphi(T,\lambda):\U(\gl_n)\overset{\sim}\longrightarrow
\U(\gl_n)$ is defined by $A\mapsto TAT^{-1}+\lambda\tr A,\ \forall
A\in \gl_n$.

 Thus, $\z'=\theta^{-1}\varphi_\lambda(\z^+)$ in that case.

  \textbf{Case 2: $\nu=-1$.} Then $\phi(y)\in V_n^*,\ \phi(x)\in V_n\ (\forall y\in
V_n,\ x\in V_n^*)$. Similarly to the above reasoning we get
$\epsilon=-1\in \mu_2, \phi(A)=-TA^tT^{-1}+\lambda\tr A\ (\forall
A\in \gl_n)$, so that there exist $\theta_1,\theta_2\in \CC^*$ such
that $\phi(y_i)=\theta_1\cdot T(x_i),\ \phi(x_j)=\theta_2\cdot
T(y_j)$. Then
$\phi(\z(y_i,x_j))=-\theta_1\theta_2\z'(T(y_j),T(x_i))$.

 Hence, $\z'=-\theta_1^{-1}\theta_2^{-1}\varphi_{-\lambda}(\z^-)$ in that case.

 Finally, the above arguments also provide isomorphisms
 $\phi_{\theta,\lambda,s}:H_{\z}(\mathfrak{gl}_n)\overset{\sim}\longrightarrow H_{\theta\varphi_\lambda(\z^s)}(\mathfrak{gl}_n)$ for any deformation
$\z$, constants $\lambda\in \CC,\theta\in \CC^*$ and a sign $s\in
\{\pm\}$.

\medskip
$\bullet$ \emph{Proof of Lemma~\ref{less_parameters}(b)}
\medskip

  Let $\z$ be a length $m$ deformation. Since $(\theta \z)_m=\theta\z_m$, we can assume $\z_m=1$.
 We claim that $\varphi_\lambda(\z)_{m-1}=0$ for $\lambda=-\z_{m-1}/(n+m)$, which is equivalent to
 $\frac{\partial{\alpha_m}}{\partial{I_n}}=(n+m)\alpha_{m-1}$.
  This equality follows from comparing coefficients of $s\tau^m$ in the identity
   $$\sum{\alpha_i(y,x)(A+sI_n)\tau^i}=(1-s\tau)^{-n-1}\sum{\alpha_i(y,x)(A)(\tau(1-s\tau)^{-1})^i}.$$

\medskip
$\bullet$ \emph{Proof of Lemma~\ref{less_parameters2}}
\medskip

  Let $\phi:H_{\z}(\spn_{2n})\overset{\sim}\longrightarrow H_{\z'}(\spn_{2n})$ be a filtration preserving isomorphism.
 Being an isomorphism, we have $\phi(1)=1$, so that $\phi$ is the identity on the $0$-th level of the filtration.

  Since $\F^{(N)}_2(H_{\z}(\spn_{2n}))=\F^{(N)}_2(H_{\z'}(\spn_{2n}))=\U(\spn_{2n})_{\leq 1}$,
 we have $\phi(A)=\psi(A)+\gamma(A)$ for all $A\in \spn_{2n}$, with $\psi(A)\in \spn_{2n}, \gamma(A)\in \CC$.
 Then $\phi([A,B])=[\phi(A),\phi(B)],\ \forall A,B\in \spn_{2n}$, if and only if $\gamma([A,B])=0$ and
 $\psi$ is an automorphism of the Lie algebra $\spn_{2n}$.
  Since $[\spn_{2n},\spn_{2n}]=\spn_{2n}$, we have $\gamma\equiv 0$.
 Meanwhile, any automorphism of $\spn_{2n}$ is inner, since $\spn_{2n}$ is a simple Lie algebra whose Dynkin diagram has no automorphisms.
 This proves $\phi_{|{\U(\spn_{2n})}}=\mathrm{Ad}(T),\ T\in
 \mathrm{Sp}_{2n}$. Composing with an automorphism $\phi'$ of $H_{\z'}(\spn_{2n})$, defined by
 $\phi'(A)=\mathrm{Ad}(T^{-1})(A),\phi'(x)=T^{-1}(x)\ (A\in \spn_{2n},x\in V_{2n})$ we can assume $\phi_{|{\U(\spn_{2n})}}=\mathrm{Id}$.

 Recall the element $I'_n=\diag(1,\ldots,1,-1,\ldots,-1)\in \spn_{2n}$.
 Since $\ad(I'_n)$ has only even eigenvalues on $U(\spn_{2n})$ and eigenvalues $\pm 1$ on $V_{2n}$,
 we actually have $\phi(V_{2n})\subset V_{2n}$.
  Together with $\End_{\spn_{2n}}(V_{2n})=\CC^*$ this implies the result.

 The converse, that is $H_{\z}(\spn_{2n})\cong H_{\theta\z}(\spn_{2n})$ for any $\z$ and $\theta\in \CC^*$, is obvious.

\medskip
\section{Minimal nilpotent case}

 We compute the isomorphism of Theorem~\ref{main 1} explicitly for the case of $e\in \g$ being the minimal
 nilpotent. This case has been considered in details in~\cite[Section 4]{P2}.

 To state the main result we introduce some more notation. Let
$z_1,\ldots,z_{2s}$ be a Witt basis of $\g(-1)$, i.e.
$\omega_{\chi}(z_{i+s},z_j)=\delta_i^j,\
\omega_{\chi}(z_i,z_j)=\omega_{\chi}(z_{i+s},z_{j+s})=0$ for any
$1\leq i,j\leq s$. We also define $\sharp:\g(0)\to \g(0)$ by
$x^\sharp:= x-\frac{1}{2}(x,h)h$. Finally, we set $c_0:=-n(n+1)/4$
for $\g=\ssl_{n+1}$ and $c_0:=-n(2n+1)/8$ for $\g=\spn_{2n}$. Then
we have the following theorem:

\begin{thm} \label{Premet_key} \cite[Theorem 6.1]{P2}
 The algebra $U(\g,e)$ is generated by the Casimir element $C$ and
the subspaces $\Theta(\zz_{\chi}(i))$ for $i=0,1$, subject to the
following relations:

$\ \ (i)\ [\Theta_x,\Theta_y]=\Theta_{[x,y]},\
[\Theta_x,\Theta_u]=\Theta_{[x,u]}$ for all $x,y\in \zz_{\chi}(0),
u\in \zz_{\chi}(1)$;

$\ \ (ii)\ C$ is central in $U(\g,e)$;

$\ \ (iii)\ [\Theta_u,\Theta_v]=
\frac{1}{2}(f,[u,v])(C-\Theta_{\Cas}-c_0)+\frac{1}{2}\sum_{1\leq
i\leq
2s}(\Theta_{[u,z_i]^{\sharp}}\Theta_{[v,z_i^*]^{\sharp}}+\Theta_{[v,z_i^*]^{\sharp}}\Theta_{[u,z_i]^{\sharp}}),$
for all $u,v\in \zz_{\chi}(1)$, where $\Theta_{\Cas}$ is a Casimir
element of the Lie algebra $\Theta(\zz_{\chi}(0))$.
\end{thm}


  Our goal is to construct explicitly isomorphisms of Theorem~\ref{main 1} for those two cases, that is, for
 $\g=\ssl_{n+1},\ \spn_{2n+2}$ and a minimal nilpotent $e\in \g$.

\begin{lem}\label{minimal 1}
 Formulas
 \begin{equation}\label{explicit1}
 \wt{\gamma}(\z_0)=\frac{c_0-C}{2},\
 \wt{\gamma}(y_i)=\Theta_{E_{i,n+1}},\
 \wt{\gamma}(x_i)=\Theta_{E_{n,i}},\
 \wt{\gamma}(A)=\Theta_A,\ A\in \gl_n\simeq \zz_\chi(0)
\end{equation}
 establish the isomorphism $H_2(\gl_{n-1})\overset{\sim}\longrightarrow U(\ssl_{n+1}, E_{n,n+1})$ from Theorem~\ref{main 1}(a).
\end{lem}

\begin{proof}
  Choose a natural $\ssl_2$-triple $(e,h,f)=(E_{n,n+1},E_{n,n}-E_{n+1,n+1},E_{n+1,n})$ in $\g=\ssl_{n+1}$.
 Then $\{E_{i,n+1}, E_{ni}\}_{1\leq i\leq n-1}$ form a basis of $\zz_{\chi}(1)$, while
 $\{E_{ij}, E_{11}-E_{kk},\ T_{n-1,2}\}_{1\leq i\ne j\leq n-1}^{2\leq k\leq n-1}$ form a basis of $\zz_{\chi}(0)$.
  Identifying $\zz_{\chi}(1)$ with $V_{n-1}\oplus V_{n-1}^*$, we get an epimorphism of algebras
 $\gamma:U(\gl_{n-1})\ltimes T(V_{n-1}\oplus V_{n-1}^*)[C]\twoheadrightarrow U(\ssl_{n+1},E_{n,n+1})$ defined by
  $$\gamma(C)=C,\gamma(y_i)=\Theta_{E_{i,n+1}}, \gamma(x_i)=\Theta_{E_{n,i}}, \gamma(I_{n-1})=\Theta_{T_{n-1,2}}, \gamma(A)=\Theta_A\ (A\in \ssl_{n-1}\subset \ssl_{n+1}).$$
 According to Theorem~\ref{Premet_key}, its kernel $\Ker(\gamma)$ is generated by

 $w\otimes w'-w'\otimes w-\frac{1}{2}(f,[\gamma(w),\gamma(w')])(C-\gamma^{-1}(\Theta_{\Cas})-c_0)-
  \gamma^{-1}(\Sym\ \sum_{1\leq i\leq2s}{\Theta_{[w,z_i]^{\sharp}}\Theta_{[w',z_i^*]^{\sharp}}}),$

\noindent
  with $w,w'\in V_{n-1}\oplus V_{n-1}^*$,\ $\gamma^{-1}(\Theta_\varsigma)\in \gl_{n-1}\oplus V_{n-1}\oplus V_{n-1}^*$
 well-defined for $\varsigma\in \zz_{\chi}(0)\oplus \zz_{\chi}(1)$.

 Choose the Witt basis of $\g(-1)$ as $z_i:=E_{i,n}, z_{i+s}:=E_{n+1,i},\
1\leq i\leq n-1=:s$.

$\bullet$
 For $w,w'\in V_{n-1}$ or $w,w'\in V_{n-1}^*$ we just get $w\otimes w'-w'\otimes w\in \Ker(\gamma)$.

$\bullet$ For $w=y_p\in V_{n-1},w'=x_q\in V_{n-1}^*$ we get the
following element of $\Ker(\gamma)$:
 $$y_p\otimes x_q-x_q\otimes y_p+\frac{\delta_p^q}{2}\left(C-\gamma^{-1}(\Theta_{\Cas})-c_0\right)-
   \gamma^{-1}(\Sym\ \sum_{1\leq i\leq 2s}{\Theta_{[E_{p,n+1},z_i]^{\sharp}}\Theta_{[E_{nq},z_i^*]^{\sharp}}}).$$
\noindent
 Let us first compute the above sum. For $1\leq i\leq s$ we obviously have $[E_{p,n+1},z_i]=0$, while
  $$[E_{p,n+1},z_{i+s}]=E_{pi}-\delta_p^iE_{n+1,n+1}\Rightarrow
    [E_{p,n+1},z_{i+s}]^{\sharp}=E_{pi}-\frac{1}{2}\delta_p^i(E_{nn}+E_{n+1,n+1}).$$
\noindent
 A similar argument implies
  $$[E_{nq},z_{i+s}^*]=E_{iq}-\delta_q^iE_{nn}\Rightarrow[E_{nq},z_{i+s}^*]^{\sharp}=E_{iq}-\frac{1}{2}\delta_q^i(E_{nn}+E_{n+1,n+1}).$$
\noindent
 Thus
  $$\Theta_{[E_{p,n+1},z_{i+s}]^{\sharp}}=\gamma(E_{pi})+\frac{1}{2}\delta_p^i\gamma(I_{n-1}),\
    \Theta_{[E_{nq},z_{i+s}^*]^{\sharp}}=\gamma(E_{iq})+\frac{1}{2}\delta_q^i\gamma(I_{n-1}),$$
\noindent
  so that
  $$\gamma^{-1}(\Sym\ \sum{\Theta_{[E_{p,n+1},z_i]^{\sharp}}\Theta_{[E_{nq},z_i^*]^{\sharp}}})=
    \Sym(\sum{E_{pi}E_{iq}})+\Sym(I_{n-1}\cdot E_{pq})+ \frac{1}{4}\delta_p^q I_{n-1}^2.$$
  On the other hand, since $\gamma^{-1}(\gamma(E_{lk})^*)=E_{kl}+\frac{1}{2}\delta_k^lI_{n-1}$, we get
    $$\gamma^{-1}(\Theta_{\Cas})=\sum_{k\ne l}{E_{kl}E_{lk}}+\sum_{k}{E_{kk}^2}+\frac{1}{2}I_{n-1}^2.$$
  Let $\wt{R}_{n-1}:=\sum{E_{ii}^2}+\frac{1}{2}\sum_{i\ne j}{(E_{ii}E_{jj}+E_{ij}E_{ji})}$. Then we get

 $y_p\otimes x_q-x_q\otimes y_p-\left(\frac{c_0-C}{2}\cdot \underbrace{\delta_p^q}_{r_0(y_p,x_q)}+
  \underbrace{\Sym(\sum{E_{pi}E_{iq}}+I_{n-1}\cdot E_{pq}+\delta_p^q\wt{R}_{n-1})}_{r_2(y_p,x_q)} \right)\in \Ker(\gamma)$.

\noindent
 This implies the statement of the lemma.
\end{proof}



\medskip

\begin{lem}\label{minimal 2}
 Formulas
\begin{equation}\label{explicit2}
 \wt{\gamma}(\xi_0)=\frac{c_0-C}{2},\
 \wt{\gamma}(y_i)=\frac{\Theta_{v_i}}{\sqrt{2}},\
 \wt{\gamma}(A)=\Theta_A,\ A\in\spn_{2n}\simeq \zz_\chi(0)
\end{equation}
 establish the isomorphism $H_1(\spn_{2n})\overset{\sim}\longrightarrow U(\spn_{2n+2},E_{1,2n+2})$ from Theorem~\ref{main 1}(b).
\end{lem}

\begin{proof}
  First, choose an $\ssl_2$-triple $(e,h,f)=(E_{1,2n+2},E_{11}-E_{2n+2,2n+2},E_{2n+2,1})$ in $\g=\spn_{2n+2}$.
 Then $\{v_k:=E_{k+1,2n+2}+(-1)^kE_{1,2n+2-k}\}_{1\leq k\leq 2n}$ form a basis of $\zz_{\chi}(1)$, while $\zz_{\chi}(0)\simeq \spn_{2n}$.
 Identifying $\zz_{\chi}(1)$ with $V_{2n}$ via $y_k\mapsto v_k$, we get an algebra epimorphism
 $$\gamma:U(\spn_{2n})\ltimes T(V_{2n})[C]\twoheadrightarrow U(\spn_{2n+2},E_{1,2n+2}),\ \ \ C\mapsto C,\ y_i\mapsto \Theta_{v_i},\ A \mapsto \Theta_A\ (A\in \spn_{2n}).$$

 According to Theorem~\ref{Premet_key}, its kernel $\Ker(\gamma)$ is
 generated by $\{y_q\otimes y_p-y_p\otimes y_q-(\ldots)\}_{p,q\leq 2n}$. Let us now compute the expression represented by the ellipsis.

  Choose the Witt basis of $\g(-1)$ with respect to the form $(a,b)=\tr(e[a,b])$ as
 $$z_i:=\frac{(-1)^{i+1}}{2}(E_{2n+2-i,1}+(-1)^iE_{2n+2,i+1}), z_{i+s}:=E_{i+1,1}-(-1)^iE_{2n+2,2n+2-i},\ 1\leq i\leq n=:s.$$
 Since $(f,[v_q,v_p])=2(-1)^q\delta_{p+q}^{2n+1}$, the above expression in ellipsis equals to:
 $$(-1)^q\delta_{p+q}^{2n+1}(C-\gamma^{-1}(\Theta_{\Cas})-c_0)+\gamma^{-1}(\Sym(\sum_{1\leq i\leq 2s}{\Theta_{[v_q,z_i]^{\sharp}}\Theta_{[v_p,z_i^*]^{\sharp}}})),$$
 where $\gamma^{-1}(\Theta_\varsigma)\in \spn_{2n}\oplus V_{2n}$ is well-defined for any $\varsigma\in \zz_{\chi}(0)\oplus \zz_{\chi}(1)$,
 though $\gamma$ is not injective.

 For any $1\leq k,l \leq 2n,\ 1\leq j\leq n$ it is easily verified that
$$[v_k,z_j]=-\frac{1}{2}(E_{k+1,j+1}-(-1)^{k+j}E_{2n+2-j,2n+2-k})-\frac{1}{2}\delta_k^j\cdot h,$$
$$[v_l,z_{j+s}]=(-1)^{j+1}(E_{l+1,2n+2-j}+(-1)^{l-j}E_{j+1,2n+2-l})+(-1)^l\delta_{l+j}^{2n+1}\cdot h,$$
so that
$$[v_k,z_j]^{\sharp}=\frac{(-1)^{k+j}E_{2n+2-j,2n+2-k}-E_{k+1,j+1}}{2},[v_l,z_{j+s}]^{\sharp}=(-1)^{j+1}E_{l+1,2n+2-j}+(-1)^{l+1}E_{j+1,2n+2-l}.$$

 We also have
 $$\gamma^{-1}(\Theta_{\Cas})=\frac{1}{4}\sum_{i,j}(E_{j,i}+(-1)^{i+j+1}E_{2n+1-i,2n+1-j})(E_{i,j}+(-1)^{i+j+1}E_{2n+1-j,2n+1-i}).$$
 On the other hand, it is straightforward to check that
$$r_0(y_q,y_p)=(-1)^p\delta_{p+q}^{2n+1},$$
$$r_2(y_q,y_p)=\frac{(-1)^{q+1}}{4}\Sym\sum_s{(E_{s,2n+1-q}+(-1)^{s+q}E_{q,2n+1-s})(E_{p,s}+(-1)^{p+s+1}E_{2n+1-s,2n+1-p})}+$$
$$\frac{(-1)^p}{8}\delta_{p+q}^{2n+1}\Sym\sum_{i,j}{(E_{i,j}+(-1)^{i+j+1}E_{2n+1-j,2n+1-i})(E_{j,i}+(-1)^{i+j+1}E_{2n+1-i,2n+1-j})}.$$

\medskip \noindent
 To summarize, the kernel of the epimorphism $\gamma$ is generated by the elements
$$\{y_q\otimes y_p-y_p\otimes y_q-(2r_2(y_q,y_p)+(c_0-C)r_0(y_q,y_p))\}_{p,q\leq 2n}.$$
 \noindent This implies the statement of the lemma.
\end{proof}

\medskip
\section{Decompositions (*) and ($\spadesuit$) for $m=-1,\ 0$}

\medskip \noindent
$\bullet$
  Decomposition isomorphism
          $H_{\hbar,-1}(\gl_n)^{\wedge_v}\cong H_{\hbar,0}^{'}(\gl_{n-1})^{\wedge_0}{\widehat{\otimes}}_{\CC[[\hbar]]} W_{\hbar,n}^{\wedge_v}$.
\medskip

  Here $H_{\hbar,0}^{'}(\gl_{n-1})$ is defined similarly to $H_{\hbar,0}(\gl_{n-1})$ with an additional central parameter $\z_0$
 and the main relation being $[y,x]=\hbar^2\z_0r_0(y,x)$, while $H_{\hbar,-1}(\gl_n):=\U_\hbar(\gl_n\ltimes (V_n\oplus  V_n^*))$.

\emph{Notation:}
  We use $y_k,\ x_l,\ e_{k,l}$ when referring to the elements of $H_{\hbar,-1}(\gl_n)$ and capital
 $Y_i,\ X_j,\ E_{i,j}$ when referring to the elements of $H_{\hbar,0}^{'}(\gl_{n-1})$.
 We also use indices $1\leq k,l\leq n$ and $1\leq i,j,i',j'<n$ to distinguish between $\leq n$ and $<n$. Finally, set $v_n:=(0,\ldots,0,1)\in V_n$.

 The following lemma establishes explicitly the aforementioned isomorphism:

\begin{lem}
 Formulas
\begin{equation*}
 \Psi_{-1}(y_k)=z_k,\
 \Psi_{-1}(e_{n,k})=z_n\partial_k,\
 \Psi_{-1}(e_{i,j})=E_{i,j}+z_i\partial_j,\
 \Psi_{-1}(e_{i,n})=z_n^{-1}Y_i\ -\sum_{j<n}{z_n^{-1}z_jE_{i,j}}+z_i\partial_n,\
\end{equation*}
\begin{equation*}
 \Psi_{-1}(x_j)=X_j,\
 \Psi_{-1}(x_n)=-z_n^{-1}\z_0-\sum_{p<n}{z_n^{-1}z_pX_p}
\end{equation*}
define an isomorphism
 $\Psi_{-1}:H_{\hbar,-1}(\gl_n)^{\wedge_{v_n}}\iso H_{\hbar,0}^{'}(\gl_{n-1})^{\wedge_0}{\widehat{\otimes}}_{\CC[[\hbar]]} W_{\hbar,n}^{\wedge_{v_n}}$.
\end{lem}

 Its proof is straightforward and is left to an interested reader
 (most of the verifications are the same as those carried out in the proof of Lemma~\ref{C.2} below).

\medskip \noindent
$\bullet$
  Decomposition isomorphism
          $H_{\hbar,0}(\gl_n)^{\wedge_v}\cong H_{\hbar,1}^{'}(\gl_{n-1})^{\wedge_0}{\widehat{\otimes}}_{\CC[[\hbar]]} W_{\hbar,n}^{\wedge_v}$.
\medskip

  Here $H_{\hbar,1}^{'}(\gl_{n-1})$ is an algebra defined similarly to $H_{\hbar,1}(\gl_{n-1})$ with an additional central parameter $\z_0$
 and the main relation being $[y,x]=\hbar^2(\z_0 r_0(y,x)+r_1(y,x))$.
 We follow analogous conventions as for variables $y_k,\ x_l,\ e_{k,l},\ Y_i,\ X_j,\ E_{i,j}$ and indices $i,j,i',j',k,l$.

 The following lemma establishes explicitly the aforementioned isomorphism:

\begin{lem}\label{C.2}
 Formulas
\begin{equation*}
 \Psi_0(y_k)=z_k,\
 \Psi_0(e_{n,k})=z_n\partial_k,\
 \Psi_0(e_{i,j})=E_{i,j}+z_i\partial_j,\
 \Psi_0(e_{i,n})=z_n^{-1}Y_i\ -\sum_{j<n}{z_n^{-1}z_jE_{i,j}}+z_i\partial_n,
\end{equation*}
\begin{equation*}
 \Psi_0(x_j)=-\partial_j+X_j,\
 \Psi_0(x_n)=-\partial_n-\sum_{i<n}{z_n^{-1}z_iX_i}-z_n^{-1}(\z_0+\sum_{i<n}E_{i,i})
\end{equation*}
 define an isomorphism
 $\Psi_0:H_{\hbar,0}(\gl_n)^{\wedge_{v_n}}\iso H_{\hbar,1}^{'}(\gl_{n-1})^{\wedge_0}{\widehat{\otimes}}_{\CC[[\hbar]]} W_{\hbar,n}^{\wedge_{v_n}}$.
\end{lem}

\begin{proof}
 These formulas provide a homomorphism
  $H_{\hbar,0}(\gl_n)^{\wedge_{v_n}}\longrightarrow H_{\hbar,1}^{'}(\gl_{n-1})^{\wedge_0}{\widehat{\otimes}}_{\CC[[\hbar]]} W_{\hbar,n}^{\wedge_{v_n}}$
 if and only if $\Psi_0$ preserves all the defining relations of $H_{\hbar,0}(\gl_n)$.
 This is quite straightforward and we present only the most complicated verifications, leaving the rest to an interested reader.

\medskip
\noindent
 $\circ$
  Verification of $[\Psi_0(e_{i,n}),\Psi_0(e_{i',j'})]=-\hbar^2\delta_{j'}^i\Psi_0(e_{i',n})$:
\medskip
 $$[\Psi_0(e_{i,n}),\Psi_0(e_{i',j'})]=[z_n^{-1}Y_i\ -\sum_{p<n}{z_n^{-1}z_pE_{i,p}}+z_i\partial_n,\ E_{i',j'}+z_{i'}\partial_{j'}]=$$
 $$\hbar^2(-\delta_{j'}^iz_n^{-1}Y_{i'}-z_n^{-1}z_{i'}E_{i,j'}+\delta_{j'}^i\sum_{p<n}{z_n^{-1}z_pE_{i',p}}+z_n^{-1}z_{i'}E_{i,j'}-\delta_{j'}^iz_{i'}\partial_n)
   =-\hbar^2\delta_{j'}^i\Psi_0(e_{i',n}).$$

\medskip
\noindent
 $\circ$
 Verification of $[\Psi_0(e_{i,n}),\Psi_0(x_j)]=-\hbar^2\delta_i^j\Psi_0(x_n)$:
\medskip
 $$[\Psi_0(e_{i,n}),\Psi_0(x_j)]=[z_n^{-1}Y_i-\sum_{q<n}{z_n^{-1}z_qE_{i,q}}+z_i\partial_n,-\partial_j+X_j]=$$
 $$-\hbar^2z_n^{-1}E_{i,j}+\delta_i^j\hbar^2\partial_n+\delta_i^j\hbar^2\sum_{q<n}{z_n^{-1}z_qX_q}+z_n^{-1}[Y_i,X_j]=$$
 $$-\hbar^2z_n^{-1}E_{i,j}+\delta_i^j\hbar^2(\partial_n+\sum_{q<n}{z_n^{-1}z_qX_q})+\hbar^2z_n^{-1}(E_{i,j}+\delta_i^j\sum_{i<n}{E_{i,i}}+\delta_i^j\z_0)=
   -\delta_i^j\hbar^2\Psi_0(x_n).$$

\medskip
\noindent
 $\circ$
  Verification of $[\Psi_0(e_{i,n}),\Psi_0(x_n)]=0$:
\medskip
 $$[\Psi_0(e_{i,n}),\Psi_0(x_n)]=[z_n^{-1}Y_i\ -\sum_{p<n}{z_n^{-1}z_pE_{i,p}}+z_i\partial_n,-\partial_n-\sum_{j<n}{z_n^{-1}z_jX_j}-z_n^{-1}(\z_0+\sum_{j<n}E_{j,j})]=$$
 $$\hbar^2(\sum_{p<n}z_n^{-2}z_pE_{i,p}-z_n^{-2}Y_i+z_iz_n^{-2}\z_0+z_iz_n^{-2}\sum_{j<n} E_{j,j}+z_n^{-2}Y_i-\sum_{j<n}z_jz_n^{-2}[Y_i,X_j])=0.$$

\medskip
 Once homomorphism $\Psi_0$ is established, it is easy to check that the map
 $$z_k\mapsto y_k,\ \partial_k\mapsto y_n^{-1}e_{n,k},\ E_{i,j}\mapsto e_{i,j}-y_iy_n^{-1}e_{n,j}, X_j\mapsto x_j+y_n^{-1}e_{n,j},$$
 $$Y_i\mapsto \sum {y_k(e_{i,k}-y_iy_n^{-1}e_{n,k})},\ \z_0\mapsto -\sum y_kx_k-\sum e_{k,k}$$
 provides the inverse to $\Psi_0$. This completes the proof of the lemma.
\end{proof}

\medskip \noindent
$\bullet$ Decomposition isomorphism
          $H_{\hbar,-1}(\spn_{2n})^{\wedge_v}\cong H_{\hbar,0}^{'}(\spn_{2n-2})^{\wedge_0}{\widehat{\otimes}}_{\CC[[\hbar]]} W_{\hbar,2n}^{\wedge_v}$.

\medskip
  Here $H_{\hbar,0}^{'}(\spn_{2n-2})$ is defined similarly to $H_{\hbar,0}(\spn_{2n-2})$ with an additional central parameter $\z_0$
 and the main relation being $[x,y]=\hbar^2\z_0r_0(x,y)$, while $H_{\hbar,-1}(\spn_{2n}):=\U_\hbar(\spn_{2n}\ltimes V_{2n})$.

\emph{Notation:}
  We use $y_k,\ u_{k,l}:=e_{k,l}+(-1)^{k+l+1}e_{2n+1-l,2n+1-k}$ when referring to the elements of $H_{\hbar,-1}(\spn_{2n})$ and
 $Y_i,\ U_{i,j}:=E_{i,j}+(-1)^{i+j+1}E_{2n-1-j,2n-1-i}$ when referring to the elements of $H_{\hbar,0}^{'}(\spn_{2n-2})$.
  Note that $\{u_{k,l}\}_{k,l\geq 1}^{k+l\leq 2n+1}$ is a basis of $\spn_{2n}$, while $\{U_{i,j}\}_{i,j\geq 1}^{i+j\leq 2n-1}$ is a basis of $\spn_{2n-2}$.
 We use indices $1\leq k,l\leq 2n$ and $1\leq i,j\leq 2n-2$.
 Finally, set $v_1:=(1,0,\ldots,0)\in V_{2n}$.

 The following lemma establishes explicitly the aforementioned isomorphism:

\begin{lem}
 Define $\psi_1(u_{k,l}):=z_k\partial_l+(-1)^{k+l+1}z_{2n+1-l}\partial_{2n+1-k}$ for all $k,l$.
 We also define
\begin{equation*}
 \psi_0(u_{1,k})=0,\
 \psi_0(u_{i+1,1})=Y_i,\
 \psi_0(u_{i+1,j+1})=U_{i,j},\
 \psi_0(u_{2n,1})=\z_0.
\end{equation*}
 Then formulas $\Upsilon_{-1}(y_k)=z_k,\ \Upsilon_{-1}(u_{k,l})=\psi_0(u_{k,l})+\psi_1(u_{k,l})$ give rise to an isomorphism
 $$\Upsilon_{-1}:H_{\hbar,-1}(\spn_{2n})^{\wedge_{v_1}}\iso H_{\hbar,0}^{'}(\spn_{2n-2})^{\wedge_0}{\widehat{\otimes}}_{\CC[[\hbar]]}W_{\hbar,2n}^{\wedge_{v_1}}.$$
\end{lem}

 The proof of this lemma is straightforward and is left to an interested reader.

\medskip \noindent
$\bullet$ Finally, we have the case of $\g=\spn_{2n},\ m=0$.

\medskip
 There is also a decomposition isomorphism
 $\Upsilon_0:H_{\hbar,0}(\spn_{2n})^{\wedge_v}\iso H_{\hbar,1}(\spn_{2n-2})^{\wedge_0}{\widehat{\otimes}}_{\CC[[\hbar]]} W_{\hbar,2n}^{\wedge_v}$.

\noindent
 This isomorphism can be made explicit, but we find the formulas quite heavy and unrevealing, so we leave them to an interested reader.


\medskip

\begin{thebibliography} {XXX}

\bibitem[BJ]{BJ}
 V.~Bavula and D.~Jordan,
   {\em Isomorphism problems and groups of automorphisms for generalized Weyl algebras},
 Trans. Amer. Math. Soc. {\bf 353} (2001), no. 2, 769--794.

\bibitem[BE]{BE}
 R.~Bezrukavnikov and P.~Etingof,
   {\em Parabolic induction and restriction functors for rational Cherednik algebras},
 Selecta Math. {\bf 14} (2009), no. 3-4, 397--425; arXiv:0803.3639.

\bibitem[BGK]{BGK}
 J.~Brundan, S.~Goodwin and A.~Kleshchev,
   {\em Highest weight theory for finite $W$-algebras},
 IMRN {\bf 15} (2008), Art. ID rnn051; arXiv:0801.1337.

\bibitem[BK1]{BK1}
 J.~Brundan and A.~Kleshchev,
   {\em Shifted Yangians and finite $W$-algebras},
 Adv. Math. {\bf 200} (2006), 136--195; arXiv:0407012.

\bibitem[BK2]{BK2}
 J.~Brundan and A.~Kleshchev,
   {\em Representations of shifted Yangians and finite $W$-algebras},
 Mem. Amer. Math. Soc. {\bf 196} (2008), no. 918, 107 pp.; arXiv:0508003.

\bibitem[BV]{BV_class}
 D. Barbasch and D. Vogan,
   {\it Primitive ideals and orbital integrals in complex classical groups},
 Math. Ann. {\bf 259} (1982), no. 2, 153--199.

\bibitem[DT]{DT}
 F.~Ding and A.~Tsymbaliuk,
   {\em Representations of infinitesimal Cherednik algebras},
 Represent. Theory {\bf 17} (2013), 557--583; arXiv:1210.4833.

\bibitem[EGG]{EGG}
 P.~Etingof, W.L.~Gan and V.~Ginzburg,
   {\em Continuous Hecke algebras},
 Transform. Groups {\bf 10} (2005), no.~3-4, 423--447; arXiv:0501192.

\bibitem[GG]{GG}
 W.L.~Gan and V.~Ginzburg,
   {\em Quantization of Slodowy slices},
 IMRN {\bf 5} (2002), 243--255; arXiv:0105225.

\bibitem[K]{K} D.~Kaledin,
 {\em Symplectic singularities from the Poisson point of view},
 J. Reine Angew. Math. {\bf 600} (2006), 135--156; arXiv:0310186.

\bibitem[Kh]{Kh}
 A.~Khare,
   {\em Category $\mathcal{O}$ over a deformation of the symplectic oscillator algebra},
 J. Pure Appl. Algebra {\bf 195} (2005), no. 2, 131--166; arXiv:0309251.

\bibitem[LNS]{LNS}
 M.~Lehn, Y.~Namikawa and Ch.~Sorger,
   {\em Slodowy slices and universal Poisson deformations},
 Compos. Math. {\bf 148} (2012), no. 1, 121--144; arXiv:1002.4107.

\bibitem[L1]{L1}
 I.~Losev,
   {\em Quantized symplectic actions and $W$-algebras},
 J. Amer. Math. Soc. {\bf 23} (2010), no. 1, 35--59; arXiv:0707.3108.

\bibitem[L2]{L5}
 I.~Losev,
   {\em Finite dimensional representations of $W$-algebras},
 Duke Math. J. {\bf 159} (2011), no. 1, 99--143; arXiv:0807.1023.

\bibitem[L3]{L6}
 I.~Losev,
   {\em On the structure of the category $\mathcal{O}$ for $W$-algebras},
 S$\acute{e}$minaires et Congr$\grave{e}$s {\bf 24} (2013),
 351--368; arXiv:0812.1584.

\bibitem[L4]{L7}
 I.~Losev,
  {\em $1$-dimensional representations and parabolic induction for $W$-algebras},
 Adv. Math. {\bf 226} (2011), no. 6, 4841--4883; arXiv:0906.0157.

\bibitem[L5]{L4}
 I.~Losev,
   {\em Completions of symplectic reflection algebras},
 Selecta. Math. {\bf 18} (2012), no. 1, 179--251; arXiv:1001.0239.

\bibitem[L6]{L8}
 I.~Losev,
   {\em Finite $W$-algebras},
 Proceedings of the International Congress of Mathematicians, Hyderabad, India, 2010,
 1281--1307; arXiv:1003.5811.

\bibitem[L7]{L2}
 I.~Losev,
   {\em Primitive ideals for W-algebras in type A},
 J. Algebra {\bf 359} (2012), 80--88; arXiv:1108.4171.


\bibitem[P1]{P1}
 A. Premet,
   {\em Special transverse slices and their enveloping algebras},
 Adv. Math. {\bf 170} (2002), no. 1, 1--55.

\bibitem[P2]{P2}
 A.~Premet,
   {\em Enveloping algebras of Slodowy slices and the Joseph ideal},
 J. Eur. Math. Soc. {\bf 9} (2007), no.~3, 487--543; arXiv:0504343.

\bibitem[T]{T}
 A.~Tsymbaliuk,
   {\em Infinitesimal Hecke algebras of $\mathfrak{so}_N$},
 arxiv:1306.1514.

\bibitem[T1]{T1}
 A.~Tikaradze,
   {\em Center of infinitesimal Cherednik algebras of $\gl_n$},
 Represent. Theory. {\bf 14} (2010), 1--8; arXiv:0901.2591.

\bibitem[T2]{T2}
 A.~Tikaradze,
   {\em On maximal primitive quotients of infinitesimal Cherednik algebras of $\gl_n$},
 J. Algebra {\bf 355} (2012), 171--175; arXiv:1009.0046.

\bibitem[T3]{T3}
 A.~Tikaradze,
  {\em Completions of infinitesimal Hecke algebras of $\ssl_2$}, arXiv:1102.1037.

\bibitem[W]{W}
 W.~Wang,
   {\em Nilpotent orbits and finite $W$-algebras},
 Fields Inst. Communications {\bf 59} (2011), 71--105; arXiv:0912.0689.
\end{thebibliography}
\end{document}